\documentclass[a4paper,12pt]{amsart}
\usepackage[T1]{fontenc}
\usepackage[cp1250]{inputenc}
\usepackage{xcolor}

\usepackage{fancybox}
\usepackage[arrow, matrix, curve]{xy}
\usepackage{amsmath,amsthm,amssymb}
\usepackage{amscd,graphics}
\usepackage[dvips]{graphicx}
\usepackage{mathtools}
\usepackage{tikz}
\tikzset{rndblock/.style={rounded corners,rectangle,draw,outer sep=0pt}}
\newcommand{\tframed}[2][]{\tikz[baseline=(h.base)]\node[rndblock,#1] (h) {#2};}

\newenvironment{Proof of}[1]{\textbf{Proof #1.}}{$\qquad \blacksquare$\par}

\DeclareMathOperator{\Aut}{Aut}
\DeclareMathOperator{\End}{End}

\DeclareMathOperator{\Ker}{Ker}

\DeclareMathOperator{\Inn}{Inn}

\newcommand{\B}{\mathcal B}

\newcommand{\X}{\widetilde X}

\newcommand{\TDelta}{\widetilde \Delta}

\newcommand{\p}{\varphi}
\newcommand{\al}{\alpha}
\newcommand{\K}{\mathcal{K}}

\newcommand{\C}{\mathbb C}
\newcommand{\R}{\mathbb R}
\newcommand{\Z}{\mathbb Z}
\newcommand{\N}{\mathbb N}
\newcommand{\T}{\mathbb T}

\newcommand{\supp}{\textrm{supp}\,}
\newtheorem{thm}{Theorem}[section]
\newtheorem{lem}[thm]{Lemma}

\newtheorem{prop}[thm]{Proposition}
\newtheorem{cor}[thm]{Corollary}
\theoremstyle{definition}
\newtheorem{defn}[thm]{Definition}
\newtheorem{ex}[thm]{Example}

\newtheorem{rem}[thm]{Remark}

\author[B. K. Kwa\'sniewski]{Bartosz Kosma  Kwa\'sniewski} 
 \address{Institute of Mathematics,  University  of Bia\l ystok\\
ul. K. Cio\l kowskiego 1M, 15-245 Bia\l ystok, Poland}
 \email{bartoszk@math.uwb.edu.pl}

\author[A. V. Lebedev]{Andrei Lebedev}   
 \address{Belorussian State University, Nesavisimosti av., 4, Minsk, Belarus;
Institute of Mathematics,  University  of Bia\l ystok\\
ul. K. Cio\l kowskiego 1M, 15-245 Bia\l ystok, Poland}
 \email{lebedev@bsu.by}

\title[Variational principles for spectral radius of
 weighted endomorphisms]{Variational principles for spectral radius of
 weighted endomorphisms of $C(X,D)$} 

\keywords{variational principle, spectral radius, endomorphism, Lyapunov exponent, cocycle}
\subjclass[2010]{47B48, 37A99 (primary), 37H15, 47A10 (secondary)} 

\thanks{The research leading to these
results has received funding from the European Union's
Seventh Framework Programme (FP7/2007-2013) under grant agreement number 621724. As well as  Polish  National Science Centre  grant number  DEC-2011/01/D/ST1/04112}
 
\begin{document}

\begin{abstract} 
We give formulas for the spectral radius of weighted endomorphisms $a\alpha: C(X,D)\to C(X,D)$, $a\in C(X,D)$, where $X$ is a compact Hausdorff space and $D$ is a unital Banach algebra.
Under the assumption that 
$\alpha$ generates a partial dynamical system $(X,\p)$, we  establish two kinds of variational principles for 
$r(a\alpha)$: using linear extensions of $(X,\p)$ and using Lyapunov exponents associated with ergodic measures for $(X,\p)$. 
This requires considering (twisted) cocycles over $(X,\p)$ with values in an arbitrary Banach algebra $D$, 
and thus our analysis can not be reduced to any of  mutliplicative ergodic theorems known so far.

The established variational principles apply not only to weighted endomorphisms but also to  a vast class of operators acting on Banach spaces that we call abstract weighted shifts
associated with $\alpha: C(X,D)\to C(X,D)$. 
 In particular,  they are far reaching generalizations of   formulas obtained by Kitover, Lebedev, Latushkin, Stepin and others.
They  are  most efficient  when $D=\B(F)$, for a Banach space $F$, and endomorphisms of $\B(F)$ induced by $\alpha$ are inner isometric. As a by product we obtain a  dynamical variational principle for an arbitrary  operator $b\in \B(F)$ and that
it's spectral radius is always a Lyapunov exponent in some direction $v\in F$, when $F$ is reflexive.

\end{abstract}

\maketitle
   
\setcounter{tocdepth}{1}
 
 \section*{Introduction }
  
	Let $\alpha:A\to A$ be an endomorphism of a Banach algebra $A$.
	The study of spectra of weighted endomorphisms $a\alpha:A\to A$,
	$a\in A$,   has a long tradition and is interesting in its own right, see, for instance,
	\cite{Kitover}, \cite{Kamowitz0}, \cite{Kamowitz}, \cite{Jamison},  where usually 
	the case when $A$ is commutative and/or $a\alpha$ is compact is considered. 
	Our interest in weighted endomorphisms stems from their relationship  with weighted composition operators. 
Spectral properties of such operators  play a crucial role in numerous problems in mathematical physics, ergodic theory, stochastic processes,  information theory, 		the theory of solvability  of functional differential equations, wavelet analysis  etc. We refer, for example,  to the books and survey articles \cite{Walters}, \cite{LatSt}, \cite{Lasota},  \cite{Anton_Lebed}, \cite{KrLit}, \cite{Anton},  \cite{ChLat}, \cite{AnBakhLeb}.
		
 In the case when the shift (the underlying map) is \emph{reversible},	
the theory of weighted composition operators was axiomatized in \cite{Anton_Lebed}.
Namely, let $A$ be a Banach subalgebra of the algebra $\B(E)$ of bounded linear operators acting on 
a Banach space $E$ and let $T\in \B(E)$ be an invertible isometry such that $TA T^{-1}=A$. 
Then   operators of the form $aT$, $a\in A$, are called   \emph{(abstract) weighted shift operators} with weight in the algebra $A$.
 Within this setting   the formula $\al (a) := TaT^{-1}$, $a\in A$, defines an automorphism $\al : A \to A$  of $A$. It turns out that all the fundamental spectral data concerning $aT$, $a\in A$, can be efficiently phrased and analyzed  in terms of  the noncommutative dynamical system $(A,\alpha)$, see \cite{Anton_Lebed}, \cite{Anton},  \cite{ChLat}, \cite{AnBakhLeb}. In particular, the spectral radii  of the weighted shift $aT\in \B(E)$ and the \emph{weighted automorphism}
 $a\alpha\in \B(A)$ coincide. 
When $A$ is a commutative uniform unital Banach algebra (i.e. the Gelfand transform $A\ni a \to \widehat{a} \in C(X)$ is an isometry) the spectral  radius is given by the following variational principle  (see   \cite{Kitover}, \cite{Lebedev79},  and 
\cite[4]{Anton_Lebed}  or \cite[5]{Anton}):
\begin{equation}
\label{e-varp-com-first}
 r(aT) =r(a\alpha)= \max_{\mu\in {\rm Inv} (X,\p)} \, \exp \int_{X} \ln|\widehat{a}(x)|\, d\mu= \max_{\mu\in {\rm Erg} (X,\p)} \, \exp \int_{X} \ln|\widehat{a}(x)|\, d\mu,
\end{equation}
where  $\varphi : X\to X$ is the homeomorphism dual to $\al$, ${\rm Inv}(X,\varphi)$ is the set of $\varphi$-invariant Borel  probability measures 
and  ${\rm Erg} (X,\p)\subseteq {\rm Inv}(X,\varphi)$ is the set of  $\p$-ergodic  measures. This covers, for instance, the situation when 
$E=L^p(X)$ or $E=C(X)$, $A$ is the algebra of operators of multiplication by functions in $C(X)$ and $T$ is an operator of composition  with a
 measure preserving homeomorphism $\p$.

We note the  resemblance of \eqref{e-varp-com-first} to the  Ruelle-Walters variational principle  \cite{Ruelle0, Ruelle89,  Walters}. The latter   
express the topological
pressure $P(\varphi, \ln|a|)$ for a continuous map $\varphi:X\to X$ and potential $\ln| a|:X\to \R\cup\{-\infty\}$, $a\in C(X)$, as 
\begin{equation}
\label{varp-topolo-pressure}
P(\varphi, \ln|a|)=\sup_{\mu\in {\rm Inv}(X,\varphi)} \left(\int_{{X}} \ln |a|\, d\mu +h_\mu(\varphi)\right),
\end{equation}
where   $h_\mu(\varphi)$ is Kolmogorov-Sinai entropy for $(X,\varphi,\mu)$. 
It covers as a particular case the Dinaburg-Goodman variational principle $h_{top}(\varphi)=\sup_{\mu\in {\rm Inv}(X,\varphi)} h_\mu(\varphi)$,
where  $h_{top}(\varphi)$ is the topological entropy of $\varphi$. Thus 
if it happens that $\varphi$ is a map of topological  entropy zero,
then \eqref{e-varp-com-first} and \eqref{varp-topolo-pressure} imply that 
the spectral exponent of $aT$ is equal to the topological pressure for $\varphi$ with potential $\ln| a|$:
$
\ln r(aT)=\ln r(a\alpha)=P(\varphi, \ln|a|).
$
We recall that  one of the main mathematical tools behind  thermodynamical formalism 
is that, in the case of topological Markov chains \cite{Ruelle0, Bowen, LatSt88}
 or, more generally, expanding local homeomorphisms \cite{Ruelle89}, topological pressure is equal to the spectral exponent of a transfer operator. 
For general irreversible dynamics that admits transfer operators, cf. \cite{kwa3}, their spectral analysis  leads to variational principles containing $t$-entropy \cite{AnBakhLeb11}.
For a comprehensive  account on  history and relationships  between the spectral analysis of weighted shift and transfer operators we refer to \cite{AnBakhLeb}.

In all the aforementioned variational principles  the algebra $A$ of weights is commutative.  
Latushkin and Stepin \cite{LatSt1} analyzed the case when $A=C(X,\B(H))$
 for a separable Hilbert space $H$ and  $\alpha:A\to A$ is given by a composition with a homeomorphism $\p:X\to X$. This 
models situation of weighted composition operators acting on vector-valued spaces $E=L^p(X, H)$.
 Under the assumption that $a\in C(X,\B(H))$ takes values in compact operators $\K(H)\subseteq \B(H)$, they proved
(see  also \cite{Anton_Lebed}  or \cite{ChLat})  that
\begin{equation}
\label{e-varp-com-second}
 \ln r(aT) =\ln r(a\alpha) = \sup_{\mu\in {\rm Erg} (X,\p)} \, \lambda_\mu,
\end{equation}
where $\lambda_\mu$ is the maximal Lyapunov exponent appearing in Ruelle's version of Multiplicative Ergodic Theorem \cite{Ruelle}
applied to the dynamical measure system $(X,\mu, \p)$ and the cocycle coming from $a:X\to \B(H)$.

The aim of the present paper is to   give a  detailed picture of the corresponding variational principles in a general \emph{irreversible}  situation and for more general \emph{non-commutative} algebras of weights. More specifically, we introduce and  initiate a study of abstract \emph{weighted shifts} $aT$, $a\in A$, \emph{associated with an endomorphism} $\alpha:A\to A$ (Definition~\ref{ban-weight-shift}). In our setting   $T\in\B(E)$ is a partial isometry on a Banach space $E$, as defined by~Mbekhta~\cite{Mbekhta}, and $\alpha(a)=TaS$, $a\in A$,  where $S\in\B(E)$ is a partial isometry  adjoint to $T$.
We note that any contractive endomorphism on a unital Banach algebra can represented in this form (Proposition~\ref{rep-delta-B}). Moreover, since  in this article  we focus on spectral radius,  our analysis 
boils down to the study of spectral radius of  the weighted endomorphism $a\alpha:A\to A$, as we always have $r(aT)=r(a\alpha)$ (see Proposition~\ref{first radius}). A number of concrete examples  of abstract weighted shifts associated with endomorphisms were considered in \cite{kwa-phd}, see also  \cite{kwa-leb08}, \cite{kwa-logist}. We plan to investigate their spectral properties in a forthcoming paper. 

In the present article we have established a number of variational principles (VPs in short). A general scheme of  relationships between them  is presented on Figure \ref{scheme_VPs}. 
\\

\begin{figure}[htb]
\begin{center} \begin{picture}(240,196)(5,-8)
\small
\put(0,183){\Ovalbox{\begin{minipage}{9cm}
\begin{center} \text{VP for $\lim\sup$ of empirical averages Theorem \ref{Var_princ_for_ergodic_sums}}\end{center}
\end{minipage}} }

\put(120,174){ \rotatebox{-90}{$\Longrightarrow$}} 
\put(-14.5,142){\footnotesize$\overbrace{\qquad\qquad \qquad\qquad\text{ VPs for spectral exponent }\qquad\qquad\qquad\qquad}^{}$}
 \put(-32,120){\Ovalbox{\begin{minipage}{11cm}\begin{center}
$\begin{array}{c| c}\,\,\text{VP using linear extension }\quad\,\, & \text{VP using Lyapunov exponents}\\ 
 \text{Theorem \ref{thm:varp-principle_for_cocycles}} & \text{Theorem  \ref{thm:variational_principle_Lyapunov}, Corollary \ref{cor:variational_principle_Lyapunov}}
\end{array}$
\end{center}
\end{minipage}} 
}

\put(59,105){ $\xymatrix{ \,\,   \ar@{=>}[dd]   \\    \,\,\,  \\ \,\, }$ }  

\put(159,105){ $\xymatrix{ \,\,   \ar@{=>}[dd]   \\    \,\,\,  \\ \,\, }$ }

\put(13,77){\scriptsize\tframed[fill=white!45]{$\begin{array}{c} r(aT)=r(a\alpha)=\lim_{n\to \infty}\|a\cdot\alpha(a)\cdot ...\cdot \alpha^n(a)\|^{1/n}
\\[2pt] \text{Proposition~\ref{first radius}}\end{array}$  }}

\put(9,35){$\overbrace{\qquad\quad\text{ VPs for weighted endomorphisms }\quad\qquad}^{}$}

 \put(17,5){\Ovalbox{\begin{minipage}{7.5cm}\begin{center}
$\begin{array}{l| l} A=C(X,\B(F))\quad   &\quad A=C(X,D) 
\\
\alpha_x\text{ inner isometric}\quad  &\quad \alpha_x \text{ arbitrary }\qquad
 \\
 \text{Theorems \ref{thm:varp-principle-TS},  \ref{spectral radius main theorem} } &\quad \text{Theorem \ref{thm:varp-principle}}
\end{array}$
\end{center}
\end{minipage}} 
}

  \end{picture} \end{center}
  \caption{Relationship between variational principles\label{scheme_VPs}}
 \end{figure}
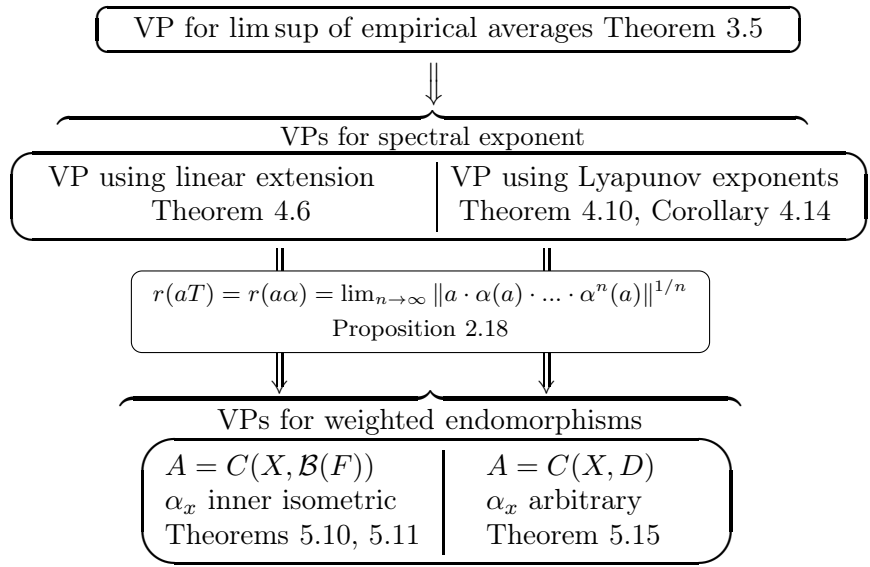
We start with preliminary Sections~\ref{s-end} and~\ref{s-weigh-sh}
where we discuss the necessary  objects and results concerning 
endomorphisms (i.e. irreversible and non-commutative dynamics) and weighted shift operators associated with endomorphisms. 
In contrast to reversible dynamics associated with automorphisms,
natural maps associated with endomorphisms are  \emph{partial  mappings}, i.e. continuous maps $\p:\Delta\to X$ 
defined on a subset $\Delta\subseteq X$ of a compact space $X$.\footnote{ 
This seemingly annoying technical detail turns out to be a friend in disguise, as when considering 'linear extensions' we need to deal with  partial maps anyway, and  we believe 
that keeping track of the domains is beneficial here.} In section \ref{erg-m-part}, we show that natural  ergodic measures  $\p:\Delta\to X$ are the usual measures for the restriction $\p:\Delta_\infty \to \Delta_\infty$ to an \emph{essential domain} of $\p$.
 Our main technical tool is what we call \emph{variational principle for $\lim\sup$ of empirical averages} over $(X,\varphi)$ (Theorem \ref{Var_princ_for_ergodic_sums}). It implies that for empirical averages the operations $\lim$ and $\sup$ do commute (Corollary~\ref{cor:ergodic_sums}). Moreover, it  readily  gives a generalization of formula \eqref{e-varp-com-first} to the case of weighted endomorphisms of a commutative uniform algebra $A$ (Theorem \ref{promień spektralny dla sumy}).

In order to deal with noncommutative Banach algebras, in Section~\ref{varp-gen-delta}, we study  Lypaunov exponents associated with an operator valued function $a:\Delta\to \B(F)$ and a partial dynamical system $(X,\p)$.
 We introduce \emph{spectral exponent} $\lambda(a,\p)$ which is equal to $\ln r(a\alpha)$ when $a\in C(X,\B(F))$ and $\alpha:C(X,\B(F))\to C(X,\B(F))$ is given by a composition with $\p$ (Definition \ref{def:spectral_exponent}). We construct a \emph{continuous linear extension} $(\X,\widetilde{\p})$ of $(X,\p)$ where $\X=X\times [B]$ and $[B]$ is a quotient of a unit ball in $F^*$.
Variational principle for $\lim\sup$ of empirical averages applied to $(\X,\widetilde{\p})$ implies  a generalization of \eqref{e-varp-com-first}  that expresses the spectral exponent $\lambda(a,\p)$ 
 in terms of  maximum of  integrals taken over the  extended system $(\X,\widetilde{\p})$ 
(Theorem \ref{thm:varp-principle_for_cocycles}).  
By projecting measures from $(\X,\widetilde{\p})$ to $(X,\p)$, the aforementioned generalization of  \eqref{e-varp-com-first}  implies a generalization of \eqref{e-varp-com-second}, 
which says that  the spectral exponent is the maximum of measure exponents  and it realizes as  Lypanuov exponent in a concrete direction when passing to dual space
(Theorem  \ref{thm:variational_principle_Lyapunov} and Corollary \ref{cor:variational_principle_Lyapunov}). 
 This result has a flavor of a variational principle for topological pressure. We note that  we achieved it without appealing to any of Mutliplicative Ergodic Theorems, cf. Remark \ref{rem:MET}. 
 
Finally, in Section \ref{sec:spectral_radius} we apply the aforementioned results to weighted endomorphisms  of $A=C(X,D)$ where $D$ is a unital Banach algebra. 
We   assume that the endomorphism satisfies $\alpha(C(X)\otimes 1)\subseteq \alpha(1)C(X)\otimes 1$, which  is equivalent to assuming that $\alpha$ is of the form
$$
\alpha(a)(x)=
\begin{cases}
\alpha_x\big(a(\p(x))\big),& x \in \Delta,\\
0,& x\notin \Delta,
\end{cases}
 \qquad  a\in C(X,D),
$$
where $(X,\p)$ is a partial dynamical system and $\{\alpha_x\}_{x\in \Delta}$ a continuous field of endomorphisms of \(D\). 
Thus not only the shift $\p$ but also the 'twist' $\{\alpha_x\}_{x\in \Delta}$ is involved. 
We show that the logarithm of spectral radius $r(a\alpha)$ is equal to a spectral exponent of a cocycle with values in  $\B(D)$.
This leads to analogues  of  \eqref{e-varp-com-first}, \eqref{e-varp-com-second} developed for general case (Theorem \ref{thm:varp-principle}).
The obtained formulas can be significantly improved when $D=\B(F)$ and all the endomorphisms $\{\alpha_x\}_{x\in \Delta}$  are isometric and inner. 
Then for any family  $\{T_x\}_{x\in \Delta}\subseteq \B(F)$ that implements endomorphisms $\{\alpha_x\}_{x\in \Delta}$, we may 
formulate generalizations of  \eqref{e-varp-com-first}, \eqref{e-varp-com-second} using cocycles associated to the function 
\begin{equation}\label{eq:discontinuous map}
\Delta \ni x \to  a(x)T_x\in \B(F),
\end{equation}
 see Theorems \ref{thm:varp-principle-TS} and  \ref{spectral radius main theorem}. One may view these last theorems as main results of the paper. 
When $\dim(F)<\infty$, they can be applied to any contractive endomorphism of $C(X,\B(F))$ 
(Remark \ref{finite-VP}). In essence, all the VPs in the paper can be deduced from these theorems. 
One of the main difficulties in proving them is that, in general, there are cohomological obstructions implying
 that \emph{the map \eqref{eq:discontinuous map} is discontinuous} 
(Remark \ref{rem:cohomological remark}). In fact, the special form of continuous linear extension $(\X,\widetilde{\p})$,
we constructed in Section~\ref{varp-gen-delta}, is dictated by the need of overcoming this difficulty.
En passant, we  mention that any operator $a\in \B(F)$ may be treated as an element in $C(\{x\},\B(F))$. Then our generalization  of \eqref{e-varp-com-first} gives a
'Dynamical Variational Principle'
 for the spectral radius of an \emph{arbitrary operator} $a\in B(F)$ (Theorem \ref{arb-VP}), and  our generalization  of \eqref{e-varp-com-second}  gives an  improvement of the classical Gelfand's formula (Corollary~\ref{cor-Gelf}).

Section \ref{sec:final} contains a brief  summary of  our result and their potential applications.


 \section{Endomorphisms of Banach algebras}
\label{s-end}

A general (irreversible) dynamics on non-commutative structures  (algebras) is naturally given by means of endomorphisms. In this  section we discuss the corresponding objects and facts that will be used in our further analysis.

\subsection{Endomorphisms of commutative algebras and partial maps}\label{endomorfizmy_algebr_Banacha1}

Let  $A$ be a  \emph{commutative} Banach algebra with an identity.\label{algebra A poczatek} Recall that the space  $X$ of non-zero linear multiplicative functionals on  $A$ equipped with  $^*$-weak topology (induced from the dual space $A^*$) is a compact Hausdorff. It is  called the  \emph{maximal ideals space} of $A$, or the  (Gelfand) \emph{spectrum} of $A$.  The \emph{Gelfand transform} is the homomorphism: 
$$
A\ni a \to \widehat{a} \in C(X),\qquad 
\widehat{a}(x) :=x(a),
$$
where  $C(X)$ is the algebra of all complex valued continuous functions on  $X$.
 Recall that  $A$ is called  
 \emph{semisimple} if the Gelfand transform is a monomorphism, that is when the  \emph{radical}  $R(A):=\bigcap_{x\in X} \Ker x$  of $A$ is zero. We say that    $A$ is  \emph{regular} if  it  is semisimple and the functions \(\widehat{a}\), \(a\in A\),
separate points from closed sets in \(X\); that is for each point   $x\in X$ and 
	 a closed subset  $F\subseteq X$ such that \(x\notin F\) there exists an element  $a\in A$ such that 
 $
 \widehat{a}(x)\neq 0$ and $\widehat{a}(F)=0$. 
\smallskip


Let  $\alpha:A \to A$ be an endomorphism. Then for every  $x\in X$ the functional   $x\circ \alpha$ is linear and multiplicative.  
Thus the dual operator to  $\alpha$ defines the map $\p$ from  $X$  to the set  $X\cup\{ 0\}$ 
(the functional  $x\circ \alpha$ may be zero). Note that the restriction of this map to $\Delta=\p^{-1}(X)$ 
takes values in  $X$ and therefore it can be 
considered a partial map on $X$.  
\begin{prop}\label{definicja alfa indukowanego} 
For any endomorphism   $\alpha:A\to A$ of a commutative Banach algebra $A$  
there is a uniquely determined continuous map  $\p:\Delta\to X$ defined on a clopen (closed and open) subset   $\Delta\subseteq X$ such that   
\begin{equation}\label{czesciowe odwz i endomor}
\widehat{\alpha(a)}(x)=\begin{cases}
\widehat{a}(\p(x)), & x\in \Delta\\
0, & x\notin \Delta
\end{cases},\qquad a\in A.
\end{equation}
Moreover,  $\Delta=X$ iff  $\alpha$ preserves the identity of algebra  $A$.
 \end{prop}
 \begin{proof} 
Uniqueness  of  $\p:\Delta\to X$ follows from  that the functions  $\widehat{a}$, $a\in A$,
 separate points of  $X$. 
Since   $\widehat{\alpha(1)}$ is the characteristic function of the set  $\Delta$, it follows that $\Delta$
is clopen and it is equal to $X$ if and only if $\alpha(1)=1$, cf.  \cite[20]{Zelazko}.  Thus we have only to verify the continuity of  $\p$ 
(note that we do not presuppose the continuity of  $\alpha$). If  $A$ is a semisimple,  then    $\alpha$ is automatically  continuous \cite[Corollary 13.4]{Zelazko}. In  general  the spectrum of quotient algebra  $A/ R(A)$, can be naturally identified with  $X$. Therefore the dual map to the homomorphism 
 $\widetilde{\alpha}: A \to A/R(A)$  given by 
 $
 \widetilde{\alpha}(a):= a + R(A)$, $a\in A$,
coincides with  the dual map to  $\alpha$. So the continuity of   $\widetilde{\alpha}$, \cite[Theorem~13.2]{Zelazko}, implies the continuity of  $\p$.   
\end{proof}
\begin{defn}\label{odwzorowania czesciowe label} 
We will call the map  $\p:\Delta\to X$ satisfying  \eqref{czesciowe odwz i endomor} 
the \emph{partial map dual to the endomorphism} $\alpha$. 
In general, by a \emph{partial dynamical system} we mean  a pair  $(X,\p)$  where $X$ is compact Hausdorff space
 and $\p:\Delta\to X$ is a continuous map defined on an open set  \(\Delta \subseteq X\).
\end{defn}
   
\subsection{Endomorphisms of $C(X,D)$}
Let \(A:=C(X,D)\) be the Banach algebra of continuous functions defined on a compact Hausdorff space  $X$ and taking values in   a   Banach algebra $D$. 
Any endomorphism $\alpha:A\to A$   gives rise to a
  family of endomorphisms $\{\alpha_x\}_{x\in X}$ of $D$ where 
\begin{equation}\label{pole generation}
\alpha_x(a):=\alpha(1\otimes a)(x), \qquad a\in D,\,\,  x\in X.
\end{equation}
 Moreover, since $\alpha(1\otimes a)\in C(X,D)$,  the mapping $X\ni x\mapsto \alpha_x(a) \in D$ is continuous for every  $a \in D$. 
 We will call $\{\alpha_x\}_{x\in X}$ the \emph{continuous field of endomorphisms of \(D\)  generated by the endomorphism $\alpha$}.
More generally, by a \emph{continuous field of endomorphism} of $D$ on $X$ we  mean any  family  $\{\alpha_x\}_{x\in X} \subseteq \End (D)$ such that  $X\ni x\mapsto \alpha_x(a) \in D$ is continuous for every  $a \in D$. 
Note that for any continuous field $\{\alpha_x\}_{x\in X} \subseteq \End (D)$ 
 the set 
\begin{equation}\label{pole generation domain}
\Delta :=\{x\in X: \alpha_x\neq 0\}
\end{equation}
is open in $X$. If, in addition, \(D\) is unital, then \(A\) is unital and \(\Delta\)   is clopen, as  we have \(\Delta=\{x\in X:\
\|\alpha(1)(x)\|=1\}\). 
\begin{prop}\label{tensor podstawowe}
Let $D$ be a Banach algebra  and let $A=C(X,D)$.  For any continuous partial map $\p:\Delta \rightarrow X$ 
 (defined on a clopen set $\Delta\subseteq X$) and any continuous field of non-zero endomorphisms $\{\alpha_x\}_{x\in \Delta}$,   the formula
 \begin{equation}\label{form of endomorphism}
\alpha(a)(x)=
\begin{cases}
\alpha_x\big(a(\p(x))\big),& x \in \Delta,\\
0,& x\notin \Delta,
\end{cases}
 \qquad  a\in A,
\end{equation}
defines an   endomorphism $\alpha$ of the algebra $A$. Formula \eqref{form of endomorphism} determines both the field of endomorphisms $\{\alpha_x\}_{x\in \Delta}$ and  the partial dynamical system
$(X,\p)$ uniquely.

Moreover, if $D$ is unital  then an arbitrary  endomorphism $\alpha$ of $A$ is of the form  \eqref{form of endomorphism} if and only if 
\begin{equation}\label{prawie zachowuje centrum}
\alpha(C(X)\otimes 1)\subseteq (C(X)\otimes 1) \cdot \alpha(1),
\end{equation}
that is,  when $\alpha$ "almost invariates" the algebra $C(X)\otimes 1$.
\end{prop}
\begin{proof}
We adopt the proof of 
\cite[Proposition 3.5]{kwa-pure}. It is obvious that the map given by \eqref{form of endomorphism} is multiplicative and linear. 
It is well defined as using the Lipschitz property of bounded linear maps we get that the map $X \times A\ni (x,a)\longmapsto  \alpha_{x}(a) \in D $ is continuous for every continuous field of endomorphisms.  
If \(\alpha\) is given by \eqref{form of endomorphism}, then both the field of endomorphisms and the set \(\Delta\) are determined by \(\alpha\) via formulae  \eqref{pole generation},  \eqref{pole generation domain}.
Now if we assume that \(\alpha\) satisfies \eqref{form of endomorphism} with \(\p\) replaced with
a different map \(\p':\Delta \to X\), then there is \(x\in \Delta\) such that \(\p(x)\neq \p'(x)\).
Take \(a\in C(X)\) such that \(a(\p(x))=1\) and  \(a(\p'(x))=0\), and let  \(b\in D\setminus \ker\alpha_x\).  On one hand we get
\(\alpha(a\otimes b)(x)=a(\p(x))\alpha_x(b)=\alpha_x(b)\neq 0\), while
  on the other \(\alpha(a\otimes b)(x)=a(\p'(x))\alpha_x(b)= 0\), a contradiction.  

Assume that $D$ is unital and $\alpha$ satisfies \eqref{prawie zachowuje centrum}. Then  for every $a\in C(X)$ there exists $a'\in C(X)$ such that  
 $$
 \al(a\otimes 1)(x)= a'(x) \al_x(1),\qquad x\in X.
 $$
Clearly, the function  $a'$ is uniquely determined  by $a$ on the  set 
$\Delta:=\{x\in X: \al(A)(x)\neq 0\}=\{x\in X: \al_x(1)\neq 0\}$. Since  $\|\alpha_x(1)\|\in \{0\}\cup [1,+\infty)$  (as a norm of an idempotent),  the set $\Delta$ is open and compact. Now it is straightforward to see that  the formula
   $
   \Phi(a)=a'|_\Delta
   $ 
defines an endomorphism $\Phi:C(X)\to C(\Delta)\subseteq C(X)$ whose range is $C(\Delta)$. 
 Hence $\Phi$ generates a partial map \(\p:\Delta \to X\), see Proposition \ref{definicja alfa indukowanego}.
For every \(a\in C(X)\) and \(b\in D\) we have 
\[\alpha(a\otimes b)(x)=\alpha(a\otimes 1) \alpha(1\otimes b)(x)=\Phi(a)(x)\alpha_x(b)=a(\p(x))\alpha_x(b)=\alpha_x(a\otimes b(\p(x))).
\]
Thus \(\alpha\) satisfies \eqref{prawie zachowuje centrum}, by continuity and linearity.
\end{proof}
 
In the nomenclature of \cite[Definition 3.3]{kwa-pure},  an endomorphism $\alpha : C(X,D) \to  C(X,D)$ satisfying  \eqref{form of endomorphism} is said to be induced by a morphism $(\p,\{\al_{x}\}_{x\in \Delta})$ of the corresponding bundle of algebras.
In accordance with Definition \ref{odwzorowania czesciowe label} we adopt the following: 
\begin{defn}\label{def:endomorphism_generate_partial_dynamical_system}
If $\alpha:C(X,D)\to C(X,D)$ is an endomorphism of the form  \eqref{form of endomorphism} we say that $\alpha$   \emph{generates the partial dynamical system $(X,\p)$}. 
\end{defn}
We will show that when $D=B(F)$ for a Banach space $F$ and the endomorphisms in the associated field are inner, then
$\alpha$ generates a partial dynamical system (see Proposition \ref{zapowiedziane twierdzenie} below). In particular, if $\dim(F)<\infty$, then every endomorphism of $C(X,\B(F))$
generates a partial dynamical system. In the infinite dimensional case,  even when $F=H$ is a Hilbert space, there are $*$-endomorphisms of  $C(X,\B(H))$ that do not generate partial dynamical systems in the sense of Definition \ref{def:endomorphism_generate_partial_dynamical_system}:
\begin{ex}
\label{endomor nie bedacy postaci}
Let $H$ be a Hilbert space and let $V_1$,{\dots}, $V_n \subseteq \B(H)$, $n>1$, be isometries with   orthogonal ranges:
 $ V^*_i V_i=1$,  $V_i^*V_j=0$, $i\neq j$. Let $X=\{x_1,x_2,{\dots},x_n\}$ be discrete space and consider  $\alpha:C(X,\B(H))\to C(X,\B(H))$ given by 
 $$
 \alpha(a)(x)=  \sum_{i=1}^n  V_i a(x_i)V_i^*, \qquad a \in C(X,\B(H)),\, x\in X.
$$
One readily sees that $\alpha$ is a  $*$-endomorphism  of $C(X,\B(H))$ that does not satisfy   \eqref{prawie zachowuje centrum}. 
\end{ex}

\subsection{Endomorphisms of $C(X,\B(F))$ generating  inner fields of endomorphisms}
\label{endom-facts}
Throughout this subsection we fix a Banach space $F$.
\begin{defn}
\label{d-inner}
An endomorphism $\alpha:\B(F)\to \B(F)$ is \emph{inner} if there  are  $T,S\in \B(F)$ such that 
$
TaS=\alpha(a) \text{ for all }a\in \B(F).
$
\end{defn}
\begin{prop}\label{prop:on_inner_endo} Suppose that  $\alpha:\B(F)\to \B(F)$ is an inner endomorphism and $T,S\in \B(F)$ are such that 
$
TaS=\alpha(a)$  for all $a\in \B(F)$.
\begin{enumerate}
\item $\alpha$ is injective if and only if $ST=1$;
\item if $\alpha$ is contractive then $\alpha$ is isometric if and only if $ST=1$ and  up to normalization $T$ is an isometry 
(i.e. for $\lambda=\|T\|^{-1}$,  we have that $\lambda T$ is  an isometry, and  $\alpha (\cdot ) = \lambda T (\cdot )\lambda^{-1}S$).
\end{enumerate}
\end{prop}
\begin{proof}
For $x\in F$ and $f\in  F^*$ we denote by $\Theta_{x,f}\in \B(F)$ the corresponding rank one operator, given by the 
formula $\Theta_{x,f}(y)=f(y)x$, $y\in F$.

(1). If $ST=1$, then $S\alpha(a)T=STa ST=a$ for all $a\in \B(F)$, and hence $\alpha$ is injective.

Conversely, assume that $ST\neq 1$. If $\alpha(1)=0$, then $\alpha$ is not injecitve. Thus we may assume that $\alpha(1)\neq 0$. Then 
$ST \neq \gamma 1$, for every $\gamma \in \C$. Indeed, if $ST = \gamma 1$, then
$$
 \alpha (1)=\alpha (1)^2=TSTS=  \alpha (ST) =\alpha (\gamma 1)    =  \gamma \alpha(1)
$$
 implies that $\gamma =1$, which  contradicts  $ST \neq \gamma 1$. Thus there exists  $x\in F$ such that $STx$ and $x$ are linearly
independent. By Hahn-Banach theorem there exists  $f\in F^*$  such that 
$f(x)= 1$ and $f(STx)=0$. Then on one hand $\Theta_{x,f}$ is a  projection onto the space spanned by $x\neq 0 $. 
On the other hand, $\alpha(\Theta_{x,f})=\Theta_{Tx,S^*f}$  is an idempotent  (since  $\Theta_{x,f}$ is an idempotent) and 
its  range is contained in the space spanned by $Tx$. However, $\Theta_{Tx,S^*f} Tx=f(STx) Tx=0$. Hence $\alpha(\Theta_{x,f})=0$ and therfore $\alpha$ fails to be injective.

(2). Let $\alpha$ be contractive. If $ST=1$ and $T$ is an isometry, 
then for every $a\in \B(F)$ we have 
$$
\|a\|=\|Ta\|=\|T a ST\|= \|\alpha(a)T\|\le\|\alpha(a)\|\le \| a\|\,.
$$
 Hence $\alpha$ is isometric. 
 
 Conversely, 
assume that $\alpha$ is isometric.  Then $ST=1$ by part (1). Let us  take any functional $f\in F^*$   of norm $1$. 
Note that $S^*f\neq 0$ since  $S$ is surjective. Now for each $h\in F$ we get  
$$
\|Th\|=\|\Theta_{Th,f}\|= \frac{\|\Theta_{Th,S^*f}\| }{\|S^*f\|}=\frac{\|\alpha(\Theta_{h,f})\|}{\|S^*f\|}
=\frac{\|\Theta_{h,f}\|}{\|S^*f\|}=\frac{\|h\|}{\|S^*f\|}.
$$ 
Thus replacing $S,T$ with $\frac{1}{\|S^*f\|}S$, $\|S^*f\|T$, we may assume that $T$ is an isometry.
\end{proof}
\begin{prop}\label{prop:inner_injective_endo} For any endomorphism  $\alpha:\B(F)\to \B(F)$ the following statements are equivalent:
\begin{enumerate}
\item $\alpha$ is injective and inner;
\item $\alpha(\B(F))=\alpha(1)\B(F)\alpha(1)$ and there is an injective $T\in \B(F)$ such that 
$$
Ta=\alpha(a)T, \qquad \text{ for every }a\in \B(F).
$$
\end{enumerate}
Moreover, if  (2) holds then there exists a unique $S\in \B(F)$ such that $\alpha(a)=TaS$ for all $a\in \B(F)$.  If $\alpha$ is isometric, we may choose $T$ to be isometry and then  $\|S\|=1$.
\end{prop}
\begin{proof}
Assume (1) and let $T,S\in B(F)$ be such that 
$
TaS=\alpha(a)$  for all $a\in \B(F)$. Then $\alpha(SaT)=TS a TS =\alpha(1)a\alpha(1)$, 
which implies $\alpha(\B(F))=\alpha(1)\B(F)\alpha(1)$ (because we always have 
$\alpha (B(F))  \subseteq \alpha (1) B(F) \alpha (1)$). 
By Proposition \ref{prop:on_inner_endo}(1) we have $ST=1$. Hence $T$ is injective and for every $a\in \B(F)$  we get $Ta=Ta ST=\alpha(a)T$.
Thus (1)$\Rightarrow$(2).

Now assume (2).  Note that for any $h\in F\setminus \{0\}$  we have $\B(F)h=F$. Hence
$$
TF=T\B(F)F=\alpha(\B(F))TF=\alpha(1)\B(F)\alpha(1) TF= \alpha(1)F.
$$
Since $\alpha(1)\in \B(F)$ is an idempotent, this implies that $TF$ is a closed complemented subspace of $F$. In particular,
$T:F\to \alpha(1)F$ is a bounded invertible operator.
Defining 
\begin{equation}
\label{e-inner-2}
Sh:=T^{-1} (\alpha(1) h), \qquad h \in F,
\end{equation}
 we get $S\in \B(F)$ such that $TS =\alpha(1)$, and therefore equality $Ta=\alpha(a)T$, $ a\in \B(F)$ implies that
$
TaS=\alpha(a)TS = \alpha (a) \alpha (1) =\alpha (a). 
$  
Hence $\alpha$ is inner. Clearly, we have $ST=1$ and threfore $\alpha$ is injective by Proposition \ref{prop:on_inner_endo}(1).
Thus (2)$\Rightarrow$(1).

Suppose now that equivalent conditions (1), (2) hold. Let $T$ be as in (2) and let $S\in \B(F)$ be any operator such that 
$
TaS=\alpha(a)$  for all $a\in \B(F)$.  Since $TS \alpha(1)=\alpha(1)\alpha(1)=\alpha(1)$, we conclude that 
$S|_{\alpha(1)F}=T^{-1}$ (recall that $T: F \to \alpha (1)F$
is a bijection). Note also that
$ST=1$ by Proposition \ref{prop:on_inner_endo}(1). Hence $
S = STS = S\alpha(1)$.  Therefore 
$S$ has to be of the form \eqref{e-inner-2}.

If $\alpha$ is isometric,  we may choose $T$ to be isometry by Proposition \ref{prop:on_inner_endo}(2). Then $\|Sh\|=\|T^{-1} (\alpha(1) h)\|=\|\alpha(1) h\|\leq \|h\|$ 
and $\|Sh\|=\|h\|$ if $h\in \alpha(1)F$. Thus $\|S\|= 1$. 
\end{proof}
\begin{ex}[endomorphisms of $M_n(\C)$]\label{ex:endomorphisms of matrices} By Skolem–Noether theorem  endomorphisms of $M_n(\C)$ are necessarily  inner automorphisms: for  every non-zero endomorphism $\alpha:M_n(\C)\to M_n(\C)$ there is an invertible $T\in M_n(\C)$ such that $\alpha(a)=TaT^{-1}$. This well known fact could also be recovered using Proposition \ref{prop:inner_injective_endo}.
\end{ex}
\begin{ex}[$*$-endomorphisms of $\B(H)$]\label{ex:endomorphisms_of_B(H)}
Let $H$ be a separable Hilbert space. Let 
  $\alpha:\B(H)\to \B(H)$ be a $*$-endomorphism.  It is well known, cf. \cite{Laca},  \cite{Brat_Jorg_Price}, that  $\alpha$ is necessarily injective. Moreover,	there is a number $n=1,2,{\dots}, \infty$  called  \emph{multiplicity index} or \emph{Powers index} for $\alpha$,
and a family $\{V_i\}_{i=1}^n$ of isometries   with orthogonal ranges
 such that
$$
 \alpha(a)= \sum_{i=1}^{n} V_i a V_i^*, \qquad a \in \B(H),
$$
where in the case $n=\infty$ the sum is weakly convergent. It follows that  a $*$-endomorphism of $B(H)$ is inner if and only if its multiplicity index is $1$ 
(the only if part follows from the last part of Lemma \ref{prop:on_inner_endo} and Proposition \ref{t-part-isom} below).
\end{ex}
Endomorphism in Example \ref{endomor nie bedacy postaci}  generates a  field of endomorphisms of $\B(H)$  with Powers index $n>1$, and thus  they are not inner. This agrees with the following:
 \begin{prop}
 \label{zapowiedziane twierdzenie}
Let  $\alpha:C(X,\B(F))\to C(X,\B(F))$ be an endomorphism and let 
$\{\alpha_x\}_{x\in \Delta}$  be the field of  (non-zero) endomorphisms of
 $\B(F)$ generated by $\alpha$. If the endomorphisms  $\{\alpha_x\}_{x\in \Delta}$  are inner, then $\alpha$ generates a  partial dynamical system.
 More specifically,  $\{\alpha_x\}_{x\in \Delta}$ are inner if and only if $\alpha$ is of the form 
 $$
\alpha(a)(x)=
\begin{cases}
T_x a(\p(x))S_x,& x \in \Delta,\\
0,& x\notin \Delta,
\end{cases}
 \qquad  a \in C(X,\B(F)),
$$ 
where $\p:\Delta \rightarrow X$ is a continuous partial map and  $\{T_x, S_x\}_{x\in \Delta}\subseteq \B(F)$. 
 \end{prop}
 \begin{proof} By Proposition  \ref{tensor podstawowe}, it suffices
 to show  that  \(\alpha\)  satisfies condition  \eqref{prawie zachowuje centrum}.
To this end, let $a\in C(X)$. We need to show that there is $a'\in C(X)$ such that
$$
\alpha(a\otimes 1)(x)=a'(x)\alpha_x(1), \qquad x \in X.
$$ 
Since we assume $\{\alpha_x\}_{x\in \Delta}$   are inner, there are operators 
$\{T_x, S_x\}_{x\in \Delta}\subseteq \B(F)$ such that $\alpha_x(a)=T_x aS_x$ for $a\in \B(F)$ and $x\in \Delta$. In particular,
we have  $\alpha_x(S_xaT_x)=\alpha_x(1) a\alpha_x(1)$. Thus putting $F_x := \alpha_x(1) F$ for each $x\in \Delta$ we may identify $\B(F_x)$
with $\alpha_x(1)\B(F)\alpha_x(1)$. Then we  get 
$
\alpha(1\otimes \B(F))(x)=\B(F_x) \subseteq \B(F)$. This implies that  $\alpha(a\otimes 1)(x)$ lies in the center $Z(\B(F_x))=\{\lambda\cdot  \alpha_x(1): \lambda \in \C\}$  of the algebra $\B(F_x)$.
Indeed,   for all $b\in \B(F)$ we have
$$
\alpha(a\otimes 1)(x) \alpha(1\otimes b)(x)=\alpha(a\otimes b)(x)= \alpha(1\otimes b)(x) \alpha(a\otimes 1)(x).
$$  
 Hence for each $x\in \Delta$ there is a number  $a'(x)\in \C$ such that
 $$
 \alpha(a\otimes 1)(x)=a'(x)\alpha_x(1).
 $$ 
The function $a':\Delta\to  \C$ obtained in this way is continuous this is forced by the continuity of the maps 
$\Delta \ni x \to  \alpha(a\otimes 1)(x)=a'(x)\alpha_x(1)$ and $\Delta\ni x \to \alpha_x(1)$). Putting $a'\equiv 0$ outside the clopen 
set \(\Delta\), we get the desired function  $a'\in C(X)$.
 \end{proof}
\begin{cor}
\label{cor-M-n}
A mapping  $\alpha: C(X,M_n(\C))\to C(X,M_n(\C))$ is an  endomorphism if and only if it is of the form
$$
\alpha(a)(x)=
\begin{cases}
T_x a(\p(x))T_x^{-1},& x \in \Delta,\\
0,& x\notin \Delta,
\end{cases}
 \qquad  a \in C(X,M_n(\C)),
$$ 
where $\p:\Delta \rightarrow X$ is a continuous partial mapping of $X$  and  $\{T_x\}_{x\in \Delta}\subseteq M_n(\C)$ is a family of invertible  matrices.
\end{cor}
\begin{proof}
Combine Proposition \ref{zapowiedziane twierdzenie} and  Example \ref{ex:endomorphisms of matrices}.
\end{proof}
 There is no a priori given universal way of choosing the operators $\{T_x, S_x\}_{x\in \Delta}\subseteq \B(F)$ in  Proposition \ref{zapowiedziane twierdzenie}. Even when considering $\B(F)$ with a weak operator topology, the mapping $X\ni x \to T_x\in \B(F)$ may be discontinuous or even unmeasurable:
\begin{ex}
\label{ex-unmeas-1}
Let $X=[0,1]$  and $\alpha_x=id$, for $x\in X$. Suppose  $V$ is a  Vitali set (an unmeasurable subset of $[0,1]$). Put
$
T_x:=1$, $S_x := 1$ for $x\in V$, and  $T_x:=-1$, $S_x:= -1$ for $x\notin V$. Then  $T:X\to \B(\C)$ is an unmeasurable field of isometries  generating continuous field of automorphisms $\alpha_x = T_x (\cdot )S_x$.
\end{ex}
In the isometric case, the choice of $\{T_x, S_x\}$ is unique up to constants in $\T$ and  we may   choose  $x \to T_x$ to be   continuous locally. Thus, the obstacles to 
continuity of $x \to T_x$ may be identified by means of cohomological objects:
\begin{lem}\label{lem:on_choice_of_S_x_T_x}
Suppose that $\{\alpha_x\}_{x\in \Delta}$ is a continuous field of  inner isometric endomorphisms of $\B(F)$. Let    $\{T_x, S_x\}_{x\in \Delta}$ be such that  
$\alpha_x(a)=T_x a S_x$, for $a\in \B(F)$,
and $T_x$ is an isometry, for $x\in \Delta$.
\begin{enumerate}
\item The operators $\{T_x, S_x\}_{x\in \Delta}$ are  determined  up to constants in $\T$.
Namely, if  $\{T_x', S_x'\}_{x\in \Delta}\subseteq \B(F)$ where  $T_x'$ are isometries, then $\alpha_x(a)=T_x' a S_x'$ for all $a\in \B(F)$ if and only if
there is $\lambda\in \T$ such that $T_x=\lambda T_x'$ and $S_x=\overline{\lambda} S_x'$.
\item For any $x_0\in \Delta$ there is an open  neighbourhood $U\subseteq \Delta$ of $x_0$ and numbers $\{\lambda_x\}_{x\in U}\subseteq \T$  
 such that  the map $U\ni x\mapsto \lambda_x T_xh\in F$ is continuous for every $h\in F$.
\end{enumerate}
\end{lem}
\begin{proof} (1). By Proposition \ref{prop:on_inner_endo}(1) we have $S_x T_x=S_x'T_x'=1$. Thus for every  $a\in \B(F)$ we get
$$
S_x T_x'a=  S_x T_x'a S_x' T_x'= S_x \alpha(a) T_x'=S_x T_x a S_x T_x'=aS_x T_x'.
$$
Hence $S_x T_x'$ belongs to the center of $\B(F)$ and therefore is a multiple  of the identity operator. That is
$S_x T_x'=\lambda 1$ for $\lambda\in \C$. This implies that
$T_x'=\lambda T_x$ (because both $T_x$, $T_x'$ are  isometries onto $\alpha_x(1)F$, and $S_x:\alpha_x(1)F\to F$ is an inverse to $T_x$).
Since $\|T_x\|=\|T_x'\|=1$  we get $\lambda\in \T$.

(2).  Take  $x_0\in X$ and $h_0\in F$. Without loss of generality we may assume $\|h_0\|=1$. Let $f\in F^*$, be such that $\|f\|=f(h_0)=1$. 
Then $\Theta_{h_0,f}$ is a norm one projection onto the space spanned by $h_0$.  Since the map $\Delta\ni x\longmapsto  \alpha_x(\Theta_{h_0,f})\in \B(F)$
is continuous the set
$$
U:=\{x\in \Delta: \alpha_x(\Theta_{h_0,f})h_0 \neq 0 \}=\{x\in \Delta:f(S_xh_0) \neq 0 \}
$$
is open. 
For every $x\in U$ 	we put 
$$
h_x:=\frac{\alpha_x(\Theta_{h_0,f})h_0}{\| \alpha_x(\Theta_{h_0,f})h_0\|}= \frac{f(S_xh_0)}{| f(S_xh_0)|}T_x h_0 
$$ so that we have $\|h_x\|=1$. 
Note that $U\ni x \longmapsto h_x \in F$ is continuous. Also  the map
$\Delta\times \B(F)\ni (x,a) \longmapsto \alpha_x(a) \in \B(F)$ is continous, because $\|\alpha_x\|=1$ for $x\in \Delta$.
Hence for every $h\in F$ the map
$$
U\ni x\longmapsto \widetilde{T}_xh := \alpha_x(\Theta_{h,f})h_x= \frac{\alpha_x(\Theta_{h,f})h_0}{\| \alpha_x(\Theta_{h_0,f})h_0\|}= \frac{f(S_xh_0)}{| f(S_xh_0)|}T_x h
$$
 is continuous. Thus putting $\lambda_x:= \frac{f(S_xh_0)}{| f(S_xh_0)|}$ we get the assertion.
\end{proof}
\begin{rem}\label{rem:cohomological remark} Let $\{\alpha_x\}_{x\in X}$ be a continuous field of $*$-automorphisms of the algebra compact operators $\K(H)$ on 
 an infinite dimensional (separable) Hilbert space $H$. Then a similar fact to Lemma \ref{lem:on_choice_of_S_x_T_x} holds, cf. \cite[Proposition 1.6]{RaeWill}. 
That is,  $\{\alpha_x\}_{x\in X}$ is implemented by a field of unitary operators $\{U_x\}_{x\in X}$, which  locally can be chosen to be continuous.
Therefore, $\{\alpha_x\}_{x\in X}$ extends uniquely to  a continuous field of $*$-automorphisms  $\B(H)$, and we may identify such fields with $C(X)$-linear automorphisms of $A:=C(X,\B(H))$. The latter form a group that we denote 
by $\Aut_{C(X)}(A)$. Let $\Inn(A)$ be the group of inner automorphisms of $A$. Clearly, $\Inn(A)$ is a  subgroup of $\Aut_{C(X)}(A)$, and  $\alpha\in\Aut_{C(X)}(A)$ can be written in the form $\alpha(a)(x)=U_x aU_x^*$, $a\in A$, for a continuous map
$X\ni x \mapsto U_x\in \B(H)$ if  and only if $\alpha \in \Inn(A)$. By \cite[Theorem 5.42]{RaeWill}  we have 
$$
\Aut_{C(X)}(A)/\Inn(A)\cong H^2(X,\Z),
$$
where  $H^2(X,\Z)$ is the second $\check{\text{C}}$ech cohomology group of $X$ with integer coefficients.
Thus whenever $H^2(X,\Z)$ is non-trivial, so for instance when $X$ is a two-dimensional sphere or a torus, 
\emph{there is always a continuous field of (inner) automorphisms $\{\alpha_x\}_{x\in X}$ of $\B(H)$ such that 
every field of operators $\{T_x\}_{x\in X}\in \B(H)$ that implements $\{\alpha_x\}_{x\in X}$ is discontinuous} (then $T_x$ is necessarily a unitary and $S_x=T_x^*$).

\end{rem}
\section{Abstract weighted shift operators and weighted endomorphisms}
\label{s-weigh-sh}

In this section we introduce abstract weighted shift operators associated with endomorphisms. To this end, we use the notion of partial isometry acting on Banach spaces in the sense of Mbekhta. We show that spectral radii of the corresponding weighted partial isometries and   weighted endomorphisms are equal. Thus the results of the present paper can be readily applied to a vast class of operators acting on Banach spaces. 

\subsection{Partial isometries on Banach spaces and endomorphisms}\label{czasciowe izometryje}

Recall that an operator  $T\in \B(H)$ acting on a Hilbert space  $H$  is  a  \emph{partial isometry} if it is an isometry on the orthogonal complement of its kernel. Then   $(\Ker T)^\bot$ is called the  \emph{initial subspace} and $T H$ the \emph{final subspace} of  $T$. Partial isometries on Hilbert spaces have a number of various well known characterisations. For instance,    $T\in \B(H)$ is a partial isometry iff one of the following equivalent conditions hold: 
\begin{itemize}
\item[i)] $T^*T$ is an orthogonal projection  (onto initial subspace),
\item[ii)] $TT^*$ is an orthogonal projection  (onto the final subspace),
\item[iii)] $TT^*T =T$,
\item[iv)] $T^*TT^* =T^*$.
\end{itemize}
We recall one more  characterisation of partial isometries which leads to a generalization of this notion to the realm of Banach spaces. 
\begin{prop}[\cite{Mbekhta} 3.1, 3.3]
\label{t-part-isom}
Let $H$ be a Hilbert space. An operator  $T\in \B(H)$  is a partial  isometry if and only if  $T$ is a contraction and there exists a contraction  $S\in \B(H)$ which is a generalized inverse to  $T$, that is 
  $
  TST=T$ and  $STS=S$
 (then  we necessarily have $S=T^*$). 
\end{prop}
\begin{defn}[\cite{Mbekhta}]\label{definicja czesciowej izometrii}  
Let  $T$ be an operator  on a Banach space $E$. We say that $T$ is a  \emph{partial isometry} if it is a contraction and there is a contraction $S\in \B(E)$  such that  
$$
TST=T,\qquad STS=S.
$$
Contractions   $T$ and $S$ satisfying  the above relations will be called mutually \emph{adjoint partial isometries}.
\end{defn}
\begin{rem}
\begin{itemize}
\item[i)] A partial isometry on a Banach space can have more than one partial isometry as adjoint, see Example~\ref{operator przesuniecia na l_p} below.
\item[ii)] Not every isometry on a Banach space is a partial isometry. On the other hand, there are  spaces (that are not Hilbert spaces) where all isometries are partial isometries.  For example,  $L^p$, $1\leq p < \infty$ are such Banach spaces, cf.  \cite{Mbekhta}. 
\end{itemize}
\end{rem}
The following proposition is  a slightly extended version of \cite[Proposition~4.2]{Mbekhta}. 
It  gives a useful description of   adjoints to a partial isometry on Banach space. 
\begin{prop}\label{description of partial isometries} 
Let  $T\in \B(E)$. The following conditions are equivalent: 
\begin{itemize}
 \item[i)] $T$ is a partial isometry,
  \item[ii)] \begin{itemize}
 \item[a)] the kernel   $\Ker T$ of  operator  $T$ possesses a complement  $M$ such that restriction of  $T$ on $M$ is an isometry,
 \item[b)] there exists a contractive projection  $P\in \B(E)$  onto the range of operator  $T$. 
  \end{itemize}
\end{itemize}
If conditions  $i)$, $ii)$ are satisfied then  relations 
$$
STE=M,\qquad TS=P
$$
establish a bijective correspondence between partial isometries    $S$ adjoint to  $T$ and pairs  $(M,P)$, where  $M$ is a complement to  $\Ker T$ and $P$ is a contractive  projection onto  $TE$. 
\end{prop}
\begin{ex}\label{operator przesuniecia na l_p} Let us consider the classical unilateral left shift operator  $T_{\N}$ acting on the space  $E=\ell^p(\N)$, $ p\in [1, \infty]$ or  $E=c_0(\N)$:
$$
T_{\N}(x(1),x(2),x(3){\dots})=(x(2),x(3),{\dots}).
$$
Clearly,  $T_\N$ is a partial isometry in the sense of Definition  \ref{definicja czesciowej izometrii}. The only projection onto  $T_\N E=E$ is the identity operator.  If   $E=\ell^p(\N)$ with  $p<\infty$, then the only complement to the subspace  $\Ker T_\N$ on which the operator  $T_\N$ is an isometry is the subspace 
$
M=\{x\in E: x(1)=0\}.
$
Therefore in this case  the only partial isometry adjoint to  $T_\N$ is the classical right shift 
$$
S_{\N}(x(1),x(2),{\dots})=(0,x(1),x(2),{\dots}).
$$
In the case when  $E=\ell^\infty(\N)$ or $E=c_0(\N)$, the situation changes. Indeed, then complements to the kernel  $T_\N$ on which the operator  $T_\N$ is an isometry can be indexed by elements of the unit ball of the dual space  $E^*$: 
$$
M_f=\{x\in l^\infty(\N): x(1)=f(x(2),x(3),{\dots})\},\qquad f \in E^*, \,\, \|f\|\leq 1.
$$
Hence  all the partial isometries  adjoint to  $T_{\N}:E \to E$ are of the form 
$$
S_f x = (f(x),x(1),x(2),{\dots}),\qquad f \in E^*. \,\, \|f\|\leq 1.  
$$
Thus, if  $E=c_0(\N)$, then partial isometries adjoint to $T_{\N}$ are indexed by  probability measures on  $\N$, while if $E=\ell^\infty(\N)$, then they are indexed by all normalized finitely additive measures on  $\N$.
\end{ex}

Let us fix a pair of mutually adjoint partial isometries  $T$, $S\in \B(E)$ acting on a Banach space  $E$. This pair naturally defines the following two mappings on  $\B(E)$: 
$$
\alpha(a):=Ta S, \qquad \alpha_*(a):=Sa T.
$$
It is straightforward to see that the mappings  $\alpha,\, \alpha_*:\B(E)\to \B(E)$ are mutually adjoint partial isometries on the Banach space  $\B(E)$.  We will discuss now some natural criteria for multiplicativity of the partial isometry $\alpha$ on subsets of  $\B(E)$. For a subset  $M\subseteq \B(E)$ we denote by  $M'$  its \emph{commutant}, that is  $M':=\{a\in \B(E): b a = ab \ \ \textrm{for each} \ \ b\in M\}$. 
 In the Hilbert space case  there is  a number of conditions equivalent to multiplicativity of $\alpha$ on $M$:
\begin{prop}\label{prop:characterizations of multiplicativity}
Let $T\in \B(H)$ be a partial isometry on a Hilbert space $H$ and put $\alpha(a):=TaT^*$ for $a\in \B(H)$. Let $M\subseteq \B(H)$ be a self-adjoint set, that is $M^*=M$.
The following conditions are equivalent:
\begin{itemize}
 \item[i)] $T^*T \in M'$,
  \item[ii)]  $Ta=\alpha(a)T$  for every   $a\in M$
	\item[iii)]   $aT^*=T^* \alpha(a)$  for every   $a\in M$.
      \item[iv)] $\alpha(ab)=\alpha(a)\alpha(b)$ for every  $a,b\in M$
 \end{itemize}
\end{prop}
\begin{proof}
ii) and iii) are equivalent, as one is  the adjoint of the other. That i) is equivalent to ii) and iii), and that they imply iv)
is easy and follows from Proposition \ref{relations commutation} below, see also \cite[Proposition 2.2]{Leb-Odz}.
To see that iv) implies ii) let $a\in M$. Then
\begin{align*}
\|Ta-\alpha(a)T\|^2&=\|Taa^*T- \alpha(a) Ta^*T^*- TaT^*\alpha(a)^* +\alpha(a)TT^*\alpha(a)^*\|
\\
&=\|\alpha(aa^*)- \alpha(a) \alpha(a^*)- \alpha(a)\alpha(a^*) +\alpha(a)\alpha(a^*)\|=0,
\end{align*}
because $a^*\in M$ and therefore $\alpha(aa^*)=\alpha(a)\alpha(a^*)$.
\end{proof}\begin{rem}
If  $M$  contains the identity operator $1\in \B(H)$, the condition iv) in Proposition \ref{prop:characterizations of multiplicativity}
implies that $T$ is necessarily a partial isometry. Indeed, if $T\in \B(H)$ is any operator such that  $\alpha(\cdot):=T(\cdot )T^*$ satisfies this condition with $a=b=1$, then 
$
TT^*=\alpha(1)=\alpha(1\cdot 1)=\alpha(1)\alpha(1)=(TT^*)^2
$ is an orthogonal projection, and hence $T$ a partial isometry.
\end{rem}
\begin{cor}
\label{conditions-endom-Hilbert}
Let  $A\subseteq \B(H)$ be a  $*$-subalgebra and $T\in \B(H)$ a partial isometry. 
The map $\alpha(a):=TaT^*$, $a\in \B(H)$, restricts to an endomorphism of $A$ if and only if
$$
TA T^* \subseteq A, \qquad T^*T \in A'.
$$
\end{cor}
\begin{proof}
The map $\alpha$ preserves $A$ if and only if $TA T^* \subseteq A$. It is multiplicative on $A$ if and only if  $T^*T \in A'$
by Proposition \ref{prop:characterizations of multiplicativity}.
\end{proof}

In the general Banach space case only some  implications in the above equivalences remain valid:

 \begin{prop}\label{relations commutation}
Let $T\in \B(E)$ be a partial isometry with an adjoint $S\in \B(E)$, and let $M\subseteq \B(H)$. Put $\alpha(a):=TaS$ for $a\in \B(E)$. The following conditions are equivalent: 
 \begin{itemize}
 \item[i)] $ST \in M'$,
  \item[ii)]  $Ta=\alpha(a)T$  and   $aS=S \alpha(a)$  for every   $a\in M$.
   \end{itemize}
Each of these equivalent conditions imply that $\alpha(ab)=\alpha(a)\alpha(b)$,  $a,b\in M$.
 \end{prop}
\begin{proof} Assuming i), for any  $a\in M$, we have $
Ta=TST a =T a ST =\alpha(a)T$ and $aS=aSTS=STaS= S \alpha(a)$.  Hence i)$\Rightarrow$ii). 
Conversely, using ii), for any  $a\in M$, we get  $
STa=S\alpha(a)T = a ST$. Thus ii)$\Rightarrow$i).  Moreover,  the definition of a partial isometry along with  i) imply that for  $a,b \in M$ we have
$
\alpha(ab)=Ta bS= TST ab S=  T a ST bT=\alpha(a)\alpha(b).
$
\end{proof}
\begin{cor}
\label{conditions-endom}
Let  $A\subseteq \B(E)$ be an algebra, and $T$ and $S$ be  mutually adjoint partial isometries such that 
\begin{equation}\label{relasio de la demokrasio}
TA S \subseteq A, \qquad ST \in A'.
\end{equation}
Then the mapping  $\alpha: A \to A$ is an endomorphism of   $A$. 
\end{cor}

For a given algebra   $A$ and a partial isometry    $T$  relations  \eqref{relasio de la demokrasio}  may be satisfied by different partial isometries  $S$ that are adjoint to  $T$, and the corresponding endomorphisms of $A$ may be different, cf. Example  \ref{different partial isometries example1} below. However, as the next statement  shows,    for  commutative algebras  and $*$-endomorphisms of $^*$-algebras,   $\alpha:A\to A$ does not depend on the choice of operator $S$ in~\eqref{relasio de la demokrasio}.

\begin{prop}\label{stwierdzenie o jednoznaczności endomorfizmu} 
Let   $A\subseteq \B(E)$ be an algebra containing $1\in \B(E)$. Let $T$ be a partial isometry and   $S_1$, $S_2\in \B(E)$ be partial isometries adjoint to  $T$  such that 
$$
TA S_i \subseteq A, \qquad S_iT \in A',\qquad \textrm{ for } \ i=1,2.
$$
Restrictions of  maps  $\alpha_i(\cdot)=T(\cdot) S_i$, $i=1,2$,  to   $A$ coincide, that is they generate the same endomorphism  $\alpha:A \to A$  if and only if  $\alpha_1(1)= \alpha_2(1)$. And   $\alpha_1(1)= \alpha_2(1)$  if and only if $\alpha_1(1)\alpha_2(1) = \alpha_2(1)\alpha_1(1)$. In particular, $\alpha_1 =\alpha_2$ on $A$, whenever  $A$ is commutative 
or  when $A$ is a $^*$-algebra and both $\alpha_i$, $i=1,2$, are $*$-preserving.
\end{prop}
\begin{proof} Note that for $a\in A$ we have
$$
\alpha_1(a)= \alpha_1(a\cdot 1)= T(a\cdot 1)S_1 = TS_2T(a\cdot 1)S_1 =  (TaS_2)(T 1S_1) = 
\alpha_2(a)\alpha_1(1).
$$
By symmetry we also get $
\alpha_2(a) = \alpha_2(a)\alpha_2(1).
$
Thus, if $\alpha_1(1) = \alpha_2(1)$ then endomorphisms $\alpha_i : A \to A$ do coincide.
The foregoing relations also show that 
\begin{equation}
\label{e.d1}
\alpha_1(1) = \alpha_2(1)\alpha_1(1)\quad \text{and}\quad \alpha_2(1) = \alpha_1(1)\alpha_2(1).
\end{equation}
Therefore $\alpha_1(1)= \alpha_2(1)$  if and only $\alpha_1(1)\alpha_2(1) = \alpha_2(1)\alpha_1(1)$.

Finally, assume that  $A$ is a  $^*$-algebra  and  
 $\alpha_i$, $i=1,2$,  are  $*$-preserving. Then   
$
\alpha_i(1) = \alpha_i (1)^*$, $i=1,2$.
These equalities along with \eqref{e.d1} give us 
$
\alpha_1 (1)  = \alpha_1 (1)^* =\alpha_1(1)^*\alpha_2(1)^* = \alpha_1(1)\alpha_2(1) =\alpha_2 (1).  
$
\end{proof}

\begin{ex}\label{different partial isometries example1}  
Let  $E=c(\N)$  be the space of converging sequences with sup-norm. Then the operator
$
T(x(1),x(2),{\dots}):= (x(1),x(1), x(2), x(3),{\dots})
$
is an isometry on 
  $ E$. Let us consider the following two contractions  $S_1$ and $S_2$ on $E$: 
$$
S_1(x(1),x(2),{\dots}):=((x(1)+x(2))/2, x(3),{\dots}),
$$
$$  
S_2(x(1),x(2),{\dots}):=(x(2),x(3),{\dots}).
$$
These are partial isometries adjoint to $T$, and
$
S_1T=S_2T=1. 
$
Thus by Corollary \ref{conditions-endom},  $\alpha_1(a):=TaS_1$ and  $\alpha_2(a):=TaS_2$ are 
endomorphisms of  $A:=B(E)$.
These endomorphisms are different, since 
$\alpha_1(1)=TS_1\neq TS_2=\alpha_2(1)$.
\end{ex}

\begin{ex}\label{different partial isometries example}
Let  $E=\ell^1(\N)$ and 
$
Tx:=(x(2)-x(1),x(3),x(4),{\dots})
$ for $x\in E$.
In this situation we have the following complements of the kernel of  $T$ on which  $T$ is an isometry:
$$
M_\lambda=\{x\in E: x(1)=-\lambda x(2)\},\quad \lambda\in [0,\infty),\qquad   M_\infty=\{( x\in E: x(1)=0 \}. 
$$
In addition $T$ is a surjection. Therefore   $T$ is a partial isometry, 
and every partial isometry $S$ adjoint to $T$ is an isometry (see Proposition \ref{description of partial isometries}). 
These operators are of the form 
$
S_{\lambda}x= \left( -x(1)/(1+\lambda), \lambda x(1)/(1+\lambda), x(2), x(3),{\dots}\right)$ 
and  $S_\infty x=(0,x(1), x(2), x(3),{\dots}), 
$
where  $\lambda\in [0,\infty)$. Since  $TS_{\lambda}=1$, $\lambda\in [0,\infty]$,  the mappings 
$$
\alpha_\lambda(a)=Ta S_\lambda,\qquad \qquad a \in \B(E),\,\, \lambda\in [0,\infty],
$$
preserve the identity  $1\in \B(E)$. Thus in view of Proposition \ref{stwierdzenie o jednoznaczności endomorfizmu} 
whenever we have a unital subalgebra $A\subseteq \B(E)$ such that a pair $(T,S_\lambda)$ satisfy relations \eqref{relasio de la demokrasio},
the restriction of $\alpha_\lambda$ to $A$ does not depend on $\lambda$.
Let us consider two situations:

i) If  $A=\{1 z: z \in \C\}$, then all the pairs   $(T,S_\lambda)$, $\lambda\in [0,\infty]$ satisfy relations  \eqref{relasio de la demokrasio} and 
 all the mappings  $\alpha_\lambda$ define the identity endomorphism on  $A$.

ii) If   $A$ is the algebra of operators of multiplication by sequences that 
are constant beginning from the second coordinate (that is sequences of the form  $(a,b,b,b,{\dots})$,  $a,b \in \C$),
 then all the mappings  $\alpha_\lambda$, $\lambda\in [0,\infty]$, preserve   $A$ while only  $\alpha_0$ and $\alpha_\infty$ are multiplicative on  $A$.  Moreover,  $\alpha_0$ and $\alpha_\infty$ define different endomorphisms on  $A$:
$
\alpha_0(a,b,b,{\dots})=(a,b,b,{\dots})$,  $\alpha_\infty(a,b,b,{\dots})=(b,b,b,{\dots}). 
$ 
This  does not contradict Proposition  \ref{stwierdzenie o jednoznaczności endomorfizmu} since among the pairs  $(T, S_\lambda)$, $\lambda\in[0,\infty]$ only  $T$ and $S_\infty$ satisfy relations  \eqref{relasio de la demokrasio}.
\end{ex}

\subsection{Weighted shift operators on Banach spaces}\label{weighted-shift}
Now we are in a position to formulate  the definition of abstract weighted shift operators associated with endomorphisms, that generalizes \cite[3.1]{Anton_Lebed} and appears in \cite{kwa-phd}:
\begin{defn}\label{ban-weight-shift}
Let $E$ be a Banach space. Suppose that $A \subseteq \B(E)$ is an algebra containing $1\in \B(E)$ and let $T\in \B(E)$ be a partial isometry which admits an adjoint   partial isometry $S$ satisfying 
$$
TA S \subseteq A,\qquad 
 ST \in A^{\prime}.
$$
So that  $\alpha : A \to A$ given by  $\alpha (a):= T(a)S $, $a\in A$,
is an endomorphism of $A$,  by Corollary~\ref{conditions-endom}. We call operators of the form
$$
 aT,\,\, a\in A,
$$
\emph{(abstract) weighted shift operators} associated with the endomorphism $\alpha$. We refer to $A$ as to the \emph{algebra of weights}. The role of shift is played by $\alpha$.
\end{defn}

\begin{rem}
\label{r-weiht}
\begin{itemize}
 \item[i)] 
Recall  that if $A$ is  commutative   then the endomorphism  $\alpha$ in this definition does not depend on the choice of $S$, and the same is true when $A$ is a $^*$-algebra and $\alpha$ is a $*$-endomorphism (see Proposition~\ref{stwierdzenie o jednoznaczności endomorfizmu}).
In these cases we will say that $T$ \emph{generates the endomorphism} $\alpha$.
 \item[ii)] If $A$ is a commutative Banach algebra then  $\alpha : A \to A$ determines a partial map 
$\p$ on the spectrum $X$ of $A$ (Proposition~\ref{definicja alfa indukowanego}). In this case we also say that  $T$ \emph{generates the   partial map} $\p$.
\end{itemize}
\end{rem}

\begin{ex}\label{operator wazonego przesuniecia na l_p}  Let  $T_{\N}$, $S_\N$ be the classical unilateral shift operators on the space  $E$ of Example~\ref{operator przesuniecia na l_p}, and let  $A\subseteq \B(E)$ be the algebra of operators of multiplication by bounded sequences: $A\cong \ell^\infty(\N)$. Then $T_\N$ and $S_\N$ are mutually adjoint partial isometries and 
$$
T_\N\,A\, S_\N \subseteq A,\qquad S_\N \,A\,T_\N \subseteq A
$$
and, in particular,
$
T_\N S_\N$,  $ S_\N T_\N \in A \subseteq A'.
$ 
Therefore the  \emph{classical unilateral weighted shift operators}\index{operator!przesunięcia z wagą!klasyczny} $aT_\N$, $aS_\N$, $a\in A$ are abstract weighted shift operators with weights in  $A\cong \ell^\infty(\N)$. Note that being an abstract weighted shift depends on the algebra $A$ we consider. Also  $aT_\N$, $aS_\N$, $a\in A$, are not abstract weighted shifts in the sense of  \cite[3.1]{Anton_Lebed}, as none of $T_\N$ and $S_\N$ is invertible.
\end{ex}

A $*$-endomorphism $\alpha:\B(H)\to \B(H)$, where $H$ is a Hilbert space, is generated by a single (partial) isometry $T\in \B(H)$ if and only if Power's index of $\alpha$ is $1$, see subsection~\ref{endom-facts}. On the other hand, it is well known, see \cite[Theorem 1.11]{kwa-leb2} or \cite[Proposition 2.6]{kwa-pure}, that any $*$-endomorphism $\alpha$  of an arbitrary
 $C^*$-algebra $A$ can be represented in a faithful and non-degenerate way on some Hilbert space, 
so that it becomes generated by a partial isometry. As we show below this result generalizes to the Banach case. 
This means that the family of weighted shift operators presented in Definition~\ref{ban-weight-shift} is vast -- in fact, up to representation,  for \emph{any contractive endomorphism} $\alpha: A\to A$ there \emph{exists a weighted shift operator} associated with it. 
Note that if $\alpha(\cdot)=T (\cdot)S$ is implemented by partial isometries $T$, $S$, it  has to be contractive,
 that is we necessarily have $\|\alpha\|\leq 1$.
\begin{prop}
\label{rep-delta-B} 
Let $A$ be a unital Banach algebra and let $\alpha:A\to A$ be a contractive endomorphism. Then there is a Banach space $E$, a unital isometric homomorphism $\pi:A\to \B(E)$ and mutually adjoint partial isometries $S$, $T\in \B(E)$ such that
$$
 \pi(\alpha (a)) =T\pi(a)S,  \text{ for }a\in A \text{ and } ST\in \pi(A)'.
$$ 
Thus for each $a \in A$ the operator $\pi(a)T$ is an abstract weighted shift  associated with the endomorphism (isometricially conjugated with) $\alpha$.
\end{prop}
 \begin{proof} We consider the Banach space 
$$
E:=\{x=(x(0),x(1),{\dots}): x(n)\in \alpha^{n}(1)A, \text{ for }n\geq 0, \text{ and } \|x\|:=\sup_{n\in \N} \|x(n)\| < \infty\}.
$$ 
Then the formula $(\pi(a)x):=\alpha^{n}(a)x(n)$  defines a unital homomorphism $\pi:A\to \B(E)$. Using that $\alpha$ is contractive and  that $\pi(a)(1,\alpha(1),\alpha^2(1),{\dots})= (a, \alpha(a),\alpha^2(a),{\dots})$ one gets that $\|\pi(a)\|=\|a\|$ for every $a\in A$. Hence 
$\pi$ is isometric.  It is readily check that putting 
$$
T(x(0),x(1),{\dots}):=(x(1),x(2),{\dots}),\quad  S(x(0),x(1),{\dots}):=(0, \alpha(1)x(0),\alpha^2(1)x(1),{\dots}),
$$
we get the desired  mutually adjoint partial isometries $S,T \in \B(E)$. 
\end{proof}
The main idea behind introducing abstract weighted shifts 
is that some of their spectral properties can be investigated in terms 
of the associated non-commutative dynamical system $(A,\alpha)$. In this paper we focus on spectral radii. 
As the next proposition shows  for these spectral characteristics the relationship between
 $aT:E\rightarrow E$ and weighted endomorphism
$a\alpha:A \rightarrow A$ is as literal as one may think. Moreover, it also tells us that  the spectral radius depends only on the values of the 'cocycle' generated by $a$ and $\alpha$, i.e.  the sequence of elements $a\alpha(a) {\dots}\, \alpha^{n}(a)$, $ n=1,2, {\dots}$ (cf. Subsection~\ref{endom-initial}).
\begin{prop}\label{first radius}
Suppose that $aT$, $a\in A$, is an  abstract weighted shift operator and $\alpha:A\to A$  is an associated endomorphism. Then for the spectral radius $r(aT)$ of the operator $aT$ we have
$$
r(aT)=r(a\alpha)=\lim_{n\to \infty}\| a\alpha(a) {\dots}\, \alpha^{n}(a)\|^{\frac{1}{n}}
$$ 
where $r(a\alpha)$ is the spectral radius of the weighted endomorphism $a\alpha:A\to A$ treated as an element of $\B(A)$.
In particular, $r(aT)= r(\alpha^k(a)T)= r(\alpha^k(a)\alpha)$ for every $k\in \N$.
\end{prop}
\begin{proof} 
 The formula  $r(a\alpha)=\lim_{n\to \infty}\| a\alpha(a) {\dots}\, \alpha^{n}(a)\|^{\frac{1}{n}}$ is a consequence  of Gelfand's formula $r(a\alpha)=\lim_{n\to \infty}\|(a\alpha)^{n}\|_{\B(A)}^{\frac{1}{n}}$ and  the following two inequalities:  
$$
\|(a\alpha)^{n+1}\|_{\B(A)}=\| a\alpha(a) {\dots}\, \alpha^{n}(a)\alpha^{n+1}\|_{\B(A)}\leq \| a\alpha(a) {\dots}\, \alpha^{n}(a)\| 
$$
and 
$$
\| a\alpha(a) {\dots}\, \alpha^{n}(a)\|\leq \| a\alpha(a) {\dots}\, \alpha^{n-1}(a)\alpha^{n}\|_{\B(A)} \|a\|.
 $$
By Proposition \ref{relations commutation}, we have $Ta= \alpha(a)T$, and hence by induction
we have $T^ka= \alpha^k(a)T^k$.  Using this and that $T$ is a contraction, we get
$$
\|(aT)^{n+1}\|=\|a\alpha(a){\dots}\alpha^{n}(a)T^{n+1} \|\leq  \|a\alpha(a){\dots}\alpha^{n}(a) \|.
$$
On the other hand, using that $T$ and $S$ are contractions we have
\begin{align*}
\|a\alpha(a){\dots}\alpha^{n}(a) \| &\leq \|a\alpha(a){\dots}\alpha^{n-1}(a)\alpha^{n}\|_{\B(A)} \|a\|
\\
&=\sup_{\|b\|=1}\|a\alpha(a){\dots}\alpha^{n-1}(a)T^{n}bS^{n }\| \|a\|
\\
&\leq \|a\alpha(a){\dots}\alpha^{n-1}(a)T^{n}\| \|a\|
\\
&=\|(aT)^{n}\| \|a\|.
\end{align*}
Thus $r(aT)=\lim_{n\to \infty}\| (aT)^n\|^{\frac{1}{n}}=\lim_{n\to \infty}\| a\alpha(a) {\dots}\, \alpha^{n}(a)\|^{\frac{1}{n}}$.

Proceeding in a similar way one can show that  
$$
\lim_{n\to \infty}\| a\alpha(a) {\dots}\, \alpha^{n-1}(a)\|^{\frac{1}{n}}=\lim_{n\to \infty}\| \alpha(a)\alpha^2(a) {\dots}\, \alpha^{n}(a)\|^{\frac{1}{n}},$$
and therefore 
$r(aT)=r(\alpha(a)T)$. Thus by induction,  $r(aT)=r(\alpha^k(a)T)$ for every  $k\in \N$.
  \end{proof}
  \begin{cor}\label{cor:spectral_radius_contractive_endomor}
	Let $\alpha:A\to A$ be a contractive endomorphism of a unital Banach algebra $A$. For any $a\in A$  and $k\in \N$  we have 
$$
r(a\alpha)= r(\alpha^k (a)\alpha) =\lim_{n\to \infty}\| a\alpha(a) {\dots}\, \alpha^{n}(a)\|^{\frac{1}{n}}.
$$
\end{cor}
\begin{proof}
By Proposition \ref{rep-delta-B} we may view $\alpha$ as being associated with an abstract weighted shift, and then the assertion follows from 
Proposition \ref{first radius}.\footnote{Alternatively, one may readilly adapt the proof of Proposition \ref{first radius}.}
\end{proof}

\section{Ergodic measures for partial maps and limits of empirical averages}
\label{erg-m-part}


We start by introducing some notation for the partial dynamical system $(X,\p)$  (cf. Definition~\ref{odwzorowania czesciowe label}).
    For 
  $n\in \N$    we denote by 
$\Delta_n$ and $\Delta_{-n}$  respectively the domain and the range of the partial map $\p^n$.  Namely, the sets  $\Delta_n$  are given by inductive formulae 
$
\Delta_0=X$,   $\Delta_n:=\p^{-1}(\Delta_{n-1})$,  $n>0$, and  then $\Delta_{-n}:=\p^{n}(\Delta_n)$.
Note that for $n>0$, the sets $\Delta_n$ are clopen while  $\Delta_{-n}$, in general, are only closed.  For all  $n,m\in \N$  one has 
$$
 \p^n:\Delta_n\rightarrow  \Delta_{-n},
$$
$$
 \p^n (\p^m(x)) = \p^{n+m}(x),\qquad x\in
\Delta_{n+m}.
$$
Note that if  $\p$ is a partial map dual to an endomorphism  $\alpha$, cf. Definition \ref{odwzorowania czesciowe label}, then   $\p^n$  is nothing but the map dual to the endomorphism   $\alpha^n$. 

\begin{defn}
 We define the  \emph{essential domain} of the partial map $\p$ as the set
$$
 \Delta_{\infty}:=\bigcap_{n\in \Z}\Delta_n.
$$
Then the  map 
$\p:\Delta_{\infty}\to \Delta_{\infty}$ is everywhere defined and surjective. 
\end{defn}

Standard definitions of invariant and ergodic measures for full maps make sense also for partial maps. Thus we define them this way. 
However, we could  equivalently define them as the corresponding notions for the restricted full map $\p:\Delta_{\infty}\to \Delta_{\infty}$.
\begin{defn}\label{essential_domain}
Let $(X,\p)$ be a partial dynamical system. Let $\mu$ be a  normalized Radon measure on $X$. We say that   $\mu$  on $X$ is \emph{$\p$-invariant}, if
$
\mu(\p^{-1}(\omega))=\mu(\omega)$, for every Borel $\omega\subseteq  X$.
If in addition for every Borel $\omega\subseteq X$   we have 
$$
\p^{-1}(\omega)=\omega \ \ \Longrightarrow \ \   \mu(\omega)=0 \ \text{or} \ \mu(\omega)=1,
$$ 
we  call $\mu$ \emph{$\p$-ergodic}. We denote by ${\rm Inv}(X,\p)$
the set of all normalized
 $\p$-invariant Radon measures on $X$, and by  ${\rm Erg}(X,\p)$ the  measures in ${\rm Inv} (X,\p)$ that are $\p$-ergodic. 
\end{defn}
\begin{lem}\label{lem1.0}
If $\mu \in {\rm Inv}(X,\p)$, then  $\supp \mu \subseteq \Delta_{\infty}$.
Thus one may consider $\p$-invariant (ergodic) measures for the partial map $\p : \Delta\to X$ as 
$\p$-invariant (ergodic) measures for full map
 $\p:\Delta_{\infty}\to \Delta_{\infty}$\,{\rm :} 
$$
 {\rm Inv}(X,\p)={\rm Inv}(\Delta_{\infty},\p), \qquad  {\rm Erg}(X,\p)={\rm Erg}(\Delta_{\infty},\p).
$$
\end{lem}

\begin{proof} By continuity of measure it suffices to show that   $\mu(\Delta_n)=1$ for every $n\in \Z$.
 We do it inductively. The zero step is obvious because  $\mu(\Delta_0)=\mu(X)=1$. 
However, if  we have  $\mu(\Delta_{k-1})=1$ for some $k>0$, then using equality $\Delta_k=\p^{-1}(\Delta_{k-1})$ and $\p$-invariance of $\mu$ we get
 $\mu(\Delta_k)=\mu(\Delta_{k-1})=1.$ Hence $\mu(\Delta_n)=1$ for all $n\geq 0$. 
\\
Now let us notice that for every $k>0$ we have $\Delta_{-k}=\p(\Delta_{-k+1}\cap \Delta)$ and therefore
$$
\mu(\Delta_{-k})=\mu\big(
\p^{-1}(\Delta_{-k})\big)=\mu\big(\p^{-1}(\p(\Delta_{-k+1}\cap
\Delta_1))\big)\geq \mu\big(\Delta_{-k+1}\cap 
\Delta_1\big)=\mu(\Delta_{-k+1})
$$
where we used $\p$-invariance of  $\mu$, the inclusion $\p^{-1}(\p(\Delta_{-k+1}\cap
\Delta_1))\supseteq \Delta_{-k+1}\cap 
\Delta_1$ and the above shown fact that $\mu(\Delta_1)=1$. Thus the assumption that $\mu(\Delta_{-k+1})=1$ implies  $\mu(\Delta_{-k})=1$. Hence by induction  we get  $\mu(\Delta_{-n})=1$  for all $n \geq 0$.
\end{proof}
\begin{rem} Recall that a point $x\in X$ is  \emph{non-wandering points} 
for a full map
$\varphi:X\to X$ if for every open  neighbourhood  $U$ of $x$
there is $n>0$  such that $\varphi^{-n}(U)\cap U\neq \emptyset$.  
Denoting the set of non-wandering points by $\Omega(\varphi)$
we  have $\varphi(\Omega(\varphi))\subseteq \Omega(\varphi)$ 
and ${\rm Inv}(X,\p)={\rm Inv}(\Omega(\varphi),\p)$.
For a partial map $\varphi:\Delta \to X$ we define
 $\Omega(\varphi):=\Omega(\varphi|_{\Delta_{\infty}})$ to be the 
set of non-wandering points for the full map $\varphi:\Delta_\infty\to \Delta_\infty$.
Then $\varphi:\Omega(\varphi)\to  \Omega(\varphi)$ is a full map and
$$
 {\rm Inv}(X,\p)={\rm Inv}(\Omega(\varphi),\p), \qquad  {\rm Erg}(X,\p)={\rm Erg}(\Omega(\varphi),\p).
$$
\end{rem}

  The next variational principle  (in the full map case) is implicit in a number of works, cf. \cite{Lebedev79}, \cite{Kitover}, \cite[4]{Anton_Lebed}, \cite[5]{Anton}. 
It implies that when considering \emph{empirical averages}, i.e.  the sums of the form \eqref{eq:ergodic_sum} below,
the operations  $\lim$ and $\sup$ commute  (see Corollary~\ref{cor:ergodic_sums}). 
\begin{thm}[Variational principle for $\lim\sup$ of empirical averages]\label{Var_princ_for_ergodic_sums}
Let $(X,\p)$ be a partial dynamical system and let $f:\Delta\to \R$ be a continuous and bounded from above function where 
 $\Delta $ is  the domain of $\varphi$ (an open subset of $X$). We define the corresponding empirical averages  to be  functions 
$S_n(f):\Delta_{n} \to \R$, $n>0$, given  by 
\begin{equation}\label{eq:ergodic_sum}
S_n(f)(x)= \frac{1}{n}(f(x) +f(\varphi(x)) +... + f(\varphi^{n-1}(x)), \qquad x \in \Delta_n.
\end{equation}
 Then 
$$
\lim_{n\to \infty} \sup_{x\in \Delta_{n}} S_n(f)(x)=\max_{\mu \in {\rm Inv}\left(\Delta_\infty,\varphi\right)} \int_{\Delta_\infty} 
f\, d\mu 
=  \max_{\mu \in {\rm Erg}\left(\Delta_\infty,\varphi\right)} \int_{\Delta_\infty} 
f \, d\mu,
$$
if ${\rm Erg}\left(\Delta_\infty,\varphi\right)\neq \emptyset$ and $
\lim_{n\to \infty} \sup_{x\in \Delta_n} S_n(f)(x)=-\infty$ otherwise.
\end{thm}

\begin{proof}
Clearly, the sequence  $a_n:=\sup_{x\in \Delta_n} (f(x) +f(\varphi(x)) +... + f(\varphi^{n-1}(x))$ is sub-additive, and 
$\sup_{x\in \Delta_n} S_n(f)(x)=\frac{a_n}{n}$.
 Hence the limit 
$\lim_{n\to \infty} \sup_{x\in \Delta_n} S_n(f)(x)$ exists  and is equal to $\inf_{n\in \N}\sup_{x\in \Delta_n} S_n(f)(x)$
(it  may be $-\infty$). For any $\mu \in {\rm Inv}\left(X,\varphi\right)$  we have 
$
\int_{\Delta_\infty} f \, d\mu = 
\int_{\Delta_\infty} f\circ \varphi \, d\mu
$ and therefore  
$$
\int_{\Delta_\infty} f \, d\mu =\lim_{n\to \infty}  \int_{\Delta_\infty} S_n(f) \, d\mu \leq \lim_{n\to \infty}  \sup_{x\in \Delta_n} S_n(f)(x) .
$$ 
To construct a measure for which the converse inequality holds, we may assume that 
$\lim_{n\to \infty}  \sup_{x\in \Delta_n} S_n(f)(x)>-\infty$. 
Then there are points $x_n\in \Delta_n$  such that 
$S_n(f)(x_n)\geq \sup_{x\in \Delta_n} S_n(f)(x)-1/n$, so that
 $\lim_{n\to \infty} \sup_{x\in \Delta_n} S_n(f)(x)=\lim_{n\to \infty}S_n(f)(x_n)$. Put
$$
\nu_n= \frac{1}{n}\sum_{i=0}^{n-1} \delta_{\varphi^i(x_n)}, \qquad n\in \N,
$$
where $\delta_x$ is the unit measure accumlated in point $x\in X$. By  Banach--Alaoglu theorem there is a subsequence  $\nu_{n_k}$ convergent in the *-weak topology to a probability measure $\nu$. In other words, 
$$
\nu(h)=\lim_{k\to \infty} \nu_{n_k}(h), \qquad h\in C(X).
$$
To prove  that $\nu$ is $\varphi$-invariant it 
 suffices to show that
$$
\nu (h \circ \p)=\nu (h), \qquad \text{for all }h\in C(X),
$$
where by $(h\circ\p)(x)$ we mean  $h(\p(x))$, when $x\in\Delta$, and  $0$ otherwise. 
However, using the definition of $\nu$ (and $\nu_{n_k}$) and  boundedness of $h\in C(X)$ we get
$$
\nu(h\circ \varphi)- \nu(h)=\lim_{k\to\infty} [\nu_{n_k}(h\circ \varphi)-\nu_{n_k}(h)]
=\lim_{k\to\infty} \frac{1}{n_k}[h(\varphi^{n_k}(x_{n_k})) -h(x_{n_k})]=0.
$$
Thus  $\nu$ is $\varphi$-invariant and in particular it supported on $\Delta_\infty$,
by Lemma \ref{lem1.0}. 

To prove the inequality $  \lim_{n\to \infty}  \sup_{x\in \Delta_n} S_n(f)(x) \leq \int_{\Delta_\infty} f \, d\mu
$ note that   in view of the choice of points $x_n$ and definition of measures $\nu_n$ we have
$$
\lim_{n\to \infty} \sup_{x\in \Delta_n} S_n(f)(x)= \lim_{k\to \infty}  S_{n_k}(f)(x_{n_k})=\lim_{k\to \infty} \nu_{n_k}(f).
$$
Even though $f$ is continuous on $\Delta$, we can not directly conclude that $\lim_{k\to \infty} \nu_{n_k}(f)=\nu(f)$, 
as $f$ may not be bounded from below. Nevertheless, putting  for    $n\in \N$ 
$$
f_n(x):=\begin{cases}
f(x), &  f(x)\geq -n \\
-n ,  &  f(x)<-n
\end{cases}
$$
we have $f
\leq f_n$ and $f_n\in C(X)$.  
Hence for each $n\in \N$
$$
\lim_{n\to \infty} \sup_{x\in \Delta_n} S_n(f)(x)=\lim_{k\to \infty} \nu_{n_k}(f)
\leq \lim_{k\to \infty} \nu_{n_k}(f_n)=\nu(f_n).
$$
Moreover, the functions  $f_n$ form a decreasing sequence that converges 
pointwise to $f$ on $\Delta$,  which is  $\nu$-almost everywhere. Thus $
\nu(f)=\lim_{n\to \infty} \nu (f_n)\geq \lim_{n\to \infty} \sup_{x\in \Delta_n} S_n(f)(x)
$.  
This concludes the proof of the equality 
$$\lim_{n\to \infty} \sup_{x\in \Delta_n} S_n(f)(x)=\max_{\mu \in {\rm Inv}\left(\Delta_\infty,\varphi\right)} \int_{\Delta_\infty}
f\, d\mu.
$$
To finish  the proof we need to show that the maximum  above is attained at an ergodic measure.
To this end take $\nu \in {\rm Inv}\left(\Delta_\infty,\varphi\right)$  such that 
 $ \int_{\Delta_\infty} f \, d\nu=\max_{\mu \in {\rm Inv}\left(\Delta_\infty,\varphi\right)} \int_{\Delta_\infty}
f\, d\mu$.
Then for every $\mu \in {\rm Erg}\left(\Delta_\infty,\varphi\right)$ we have   
$$
 \int_{\Delta_\infty} f\, d\nu \ge  \int_{\Delta_\infty} f\, d\mu. 
$$
By  Choquet-Bishop-de Leeuw Theorem (see, for instance, \cite[page 22]{Phelps}), there exists  a probability measure
$m$ on ${\rm Erg}\left(\Delta_\infty,\varphi\right)$  such that
$$
 \int_{\Delta_\infty} f\, d\nu  =  \int_{{\rm Erg}\left(\Delta_\infty,\varphi\right)} \left(\int_{\Delta_\infty} f\, d\mu\right)\,\, dm (\mu).
$$
The above equality and the earlier inequality  imply existence of  $\mu \in{\rm Erg}\left(\Delta_\infty,\varphi\right)$ with 
$
 \int_{\Delta_\infty} f\, d\nu=  \int_{\Delta_\infty} f\, d\mu. 
$
\end{proof}
\begin{cor}\label{cor:ergodic_sums}
Retain the notation and assumptions of Theorem \ref{Var_princ_for_ergodic_sums}. There is $\mu \in {\rm Erg}\left(\Delta_\infty,\p\right)$
and a subset $Y\subseteq \Delta_\infty$, $\mu (Y)=1$, such that 
$$
\lim_{n\to \infty} S_n(f)(y)=\lim_{n\to \infty} \sup_{x\in \Delta_n} S_n(f)(x) \qquad \text{ for every }y\in Y.
$$
In particular, $\lim_{n\to \infty} \sup_{x\in \Delta_n} S_n(f)(x)=\max_{x\in \Delta_\infty} \lim_{n\to \infty} S_n(f)(x)$.
\end{cor}
\begin{proof}
By Theorem \ref{Var_princ_for_ergodic_sums} there is $\mu \in {\rm Erg}\left(\Delta_\infty,\p\right)$ such that 
 $\lim_{n\to \infty} \sup_{x\in \Delta_n} S_n(f)(x)=\int_{\Delta_\infty} f \, d\mu
  $.   By the Birkhoff-Khinchin ergodic theorem there exists a subset  $Y\subseteq \Delta_\infty$,  $\mu (Y)=1$ such that for every $y \in Y$  the 
sum $
S_n(f)(y)$ converges to  $\int_{\Delta_\infty} 
f \, d\mu$. This gives the first part of assertion. 

For the second part note  that for every $x\in \Delta_\infty$,   the sequence  $a_n:=(f(x) +f(\varphi(x)) +... + f(\varphi^{n-1}(x))$ is (sub-)additive, and 
therefore the limit of $ S_n(f)(x)=\frac{a_n}{n}$ exists. Moreover, we clearly have   
$\sup_{x\in \Delta_\infty} \lim_{n\to \infty} S_n(f)(x) \leq \lim_{n\to \infty} \sup_{x\in \Delta_n} S_n(f)(x)$. This together with the first part of the assertion gives the 
desired equality.
\end{proof}
\section{Variational principles  for cocycles and Lyapunov exponents}
\label{varp-gen-delta}

Here we introduce the spectral exponent of an operator-valued function $a:\Delta \to \B(F)$  and prove  variational principles that express this exponent either in terms of a linear extension or  in terms of measure Lyapunov exponents. These results will serve  as fundamental  instruments in the proofs of all  the variational principles for  spectral radius of weighted endomorphisms discussed further in Section~\ref{sec:spectral_radius}.

\subsection{Cocycles, Lyapunov exponents and linear extensions}
\label{endom-initial}

Let us fix a partial dynamical system $(X,\p)$  and a Banach space $F$.   Let us also fix an operator valued function $\Delta\ni x \to a(x) \in \B(F)$ which is bounded in the sense that $\sup_{x\in \Delta}\|a(x)\|<\infty$. 
\begin{defn} We associate to the triple $(X,\p,a)$ two functions 
$C_{a,\p}^f, C_{a,\p}^b:\{(x,n)\in X\times \N: x\in \Delta_n\} \to \B(F)$ given  by the formulae
\begin{align*}
C_{a,\p}^f(x,n)&:=a(x)\cdot a(\p(x))\cdot{\dots} \cdot  
a(\p^{n-1}(x)),
\\
C_{a,\p}^b(x,n)&:=a(\p^{n-1}(x))\cdot{\dots} \cdot a(\p(x))\cdot a(x),  
\end{align*}
where  $x\in \Delta_{n}$,  $n\in \N$. We call $C_{a,\p}^f$ and $C_{a,\p}^b$ \emph{cocycles of $a$ with respect to $\p$}. We refer to $C_{a,\p}^f$ as a \emph{forward cocycle}  and  to $C_{a,\p}^b$ as a \emph{backward cocycle}. 
\end{defn}
We are interested in ergodic properties of these cocycles and the arising variational principles for them. If $\p:X\to X$ is a homeomorphism we have
$$
C_{a,\p}^f(x,n)=  C_{a,\p^{-1}}^b(\p^{n-1}(x),n),\qquad C_{a,\p}^b(x,n)=  C_{a,\p^{-1}}^f(\p^{n-1}(x),n),
$$ 
so  we can study properties of the forward cocycle by looking at the backward cocycle, and vice versa.
In general, in the  irreversible case,  we can relate the forward and backward cocycles by passing to adjoints.
Namely, let $F^*$ be the space dual to $F$ and  define the adjoints  $C_{a,\p}^{*f}, C_{a,\p}^{*b}:\{(x,n)\in X\times \N: x\in \Delta_n\} \to \B(F^*)$
of $C_{a,\p}^f$, $C_{a,\p}^b$ by taking pointwise adjoints: $C^*_{a,\alpha}(x,n):=C_{a,\alpha}(x,n)^*$ and  $C^*_{aT, \varphi}(x,n):=C_{aT, \varphi}(x,n)^*$. Similarly, define $\Delta\ni x \mapsto a^*(x):=a(x)^*\in \B(F^*)$. Then we have 
\begin{equation}\label{eq:cocycles_vs_adjoints}
C_{a,\p}^{f*}=C_{a^{*},\p}^b \quad\textrm{ and }\quad C_{a,\p}^{b*}=C_{a^{*},\p}^f.
\end{equation}
Both $C_{a,\p}^b$ and $C_{a,\p}^{f}$ have their advantages. Ergodic properties of $C_{a,\p}^b$ seem easier to calculate, while $C_{a,\p}^{f}$ is more relevant for calculation of spectral radius:
\begin{defn}\label{def:spectral_exponent}
Since the sequence $a_n:=\sup_{x\in \Delta_n} \ln \|C_{a,\p}^f(x,n)\|$ is subadditive (i.e. $a_{m+n} \le a_m + a_n$), we have the  following equality
$$
\lambda(a,\p):=\lim_{n\to \infty}\sup_{x\in \Delta_n} \frac{1}{n} \ln \|C_{a,\p}^f(x,n)\|=\inf_{n\in \N}\sup_{x\in \Delta_n} \frac{1}{n} \ln \|C_{a,\p}^f(x,n)\|.
$$
We call its common value $\lambda(a,\p)\in [-\infty,\infty)$ the  \emph{spectral exponent} of $a:\Delta \to \B(F)$ with respect to $(X,\p)$. 
\end{defn}
\begin{rem}\label{rem:spectral_exponents_and_radii} Let $D\subseteq \B(F)$ be  Banach algebra, and let \(\alpha:C(X,D)\to C(X,D)\) be a contractive endomorphism that generates a partial dynamical system $(X,\p)$.  For every $a\in D\subseteq \B(F)$ we define the function
$C_{a,\alpha}\{(x,n)\in X\times \N: x\in \Delta_n\} \to \B(F)$ by the formula
$$
C_{a,\alpha}(x,n):=a\cdot\alpha(a)\cdot {\dots} \cdot \alpha^{n-1} (a) (x)\qquad  \text{for   $x\in \Delta_{n}$}.
$$
It follows from Corollary \ref{cor:spectral_radius_contractive_endomor} that
$$
r(a\alpha)=
\lim_{n\to \infty}\max_{x\in \Delta_n} \|C_{a,\alpha}(x,n)\|^{\frac{1}{n}}.
$$
In particular, if the field of endomorphisms $\{\alpha_x\}_{x\in \Delta}$ generated by $\alpha$ is trivial, i.e $\alpha_x\equiv id_{D}$, $x\in \Delta$, then $C_{a,\alpha}=C_{a,\varphi}^f$ and therefore
$$
\ln r(a\alpha) = \lambda(a,\varphi).
$$
This motivates Definition \ref{def:spectral_exponent}.  In addition it  also indicates why in the case when the field  $\{\alpha_x\}_{x\in \Delta}$ is non-trivial, deriving formulas for  $r(a\alpha)$ is much harder and requires extra work and this will be our aim.
\end{rem}

Let $B$ be the unit ball in the dual space $F^*$ equipped with $^*$-weak topology. In order to deal with the potential discontinuity of  $a:\Delta \to \B(F)$, 
cf. Example \ref{ex-unmeas-1} and Remark \ref{rem:cohomological remark}, we will consider a quotient of $B$ by the following equivalence relation:
$$
v\sim w\,\,\Longleftrightarrow\,\, v=\lambda w \text{ for some }\lambda\in \T.
$$
Note that $v\sim w$ if and only if $|v|=|w|$  as functions. 
We will write $[v]$ for the equivalence class of $v\in B$. Thus $[v]=\{w\in B: |v|=|w|\}$. In what follows  we denote by $[B]$ the factor space $B/\sim$. 
\begin{lem}\label{lem:projective_ball}
The space $[B]$ is compact and Hausdorff.
\end{lem}
\begin{proof} $[B]$ is compact as a continuous image of the compact space $B$. 
Now,  take any $[v]$, $[w]\in [B]$ with $[v]\neq [w]$. Then there is $h\in F$ such that $|v(h)|\neq|w(h)|$. We may
assume that $|v(h)|<|w(h)|$. Take $\varepsilon=\frac{|w(h)|-|v(h)|}{2}$ and put
$$
V:=\{[u]\in [B]: |\,|v(h)|-|u(h)|\,|<\varepsilon\},\qquad W:=\{[u]\in [B]: |\,|w(h)|-|u(h)|\,|<\varepsilon\}. 
$$
Then $V$ and $W$ are open neighbourhoods of $[v]$ and $[w]$ in $[B]$.
Moreover, if we assume that $[u]\in V\cap W$, then there are $\lambda_v$, $\lambda_w\in \T$ such
that $|\lambda_v v(h)-u(h)|, |\lambda_w w (h)-u(h)| <\varepsilon$. This implies 
\begin{align*}
 |w(h)|- |v(h)|&=|\lambda_w w(h)|- |\lambda_v v(h)|\leq |\lambda_w w(h)-\lambda_v v(h)|
\\
&\leq |\lambda_v v(h)-u(h)|+ |\lambda_w w (h)-u(h)|
< 2\varepsilon=|w(h)|- |v(h)|,
\end{align*}
a contradiction. Hence $V$ and $W$ are disjoint, and $[B]$ is Hausdorff.
\end{proof}
By Lemma \ref{lem:projective_ball} the space $\widetilde{X}:= X\times {[B]}$ equipped with the product topology is a compact Hausdorff space.
Construction described in the next definition plays the principal role in variational principles under study.
\begin{defn}\label{defn:linear_extension}
We say that a triple $(X,\p,a)$  \emph{admits a continuous linear extension} $(\X,\tilde{\varphi}, \tilde{a})$ if the set
$$
\TDelta:= \{(x,[v])\in \widetilde{X}: \, x\in \Delta,\,\,  a^{*}(x)v  \neq 0\},
$$
is open in $\X$, and the maps $\tilde{\varphi}:\TDelta \to \widetilde{X}$  and $\tilde{a}: \TDelta\to [0,+\infty)$ given by
\begin{equation}
\label{e-lin-ext1}
\tilde{\p} (x,[v]):= \left(\varphi (x), \left[\frac{a^*(x)v}{\|a^*(x)v\|}\right]\right), \qquad \tilde{a}(x,[v]):= 
\| a^*(x)v\|
\end{equation} 
are  continuous. 
\end{defn} 
Clearly, $(X,\p,a)$ admits a continuous extension whenever $a$ is continuous, that is when $a\in C(\Delta,\B(F))$. However,   it may occur that   $(X,\p,a)$ admits a continuous extension even   when $a$ is  not  measurable: for instance, take $a=T$ in Example \ref{ex-unmeas-1}. 
\begin{thm}[Variational principle using linear extension]
\label{thm:varp-principle_for_cocycles} 
Let $(X,\p)$ be a partial dynamical system and  let $\Delta\ni x \to a(x) \in \B(F)$ a bounded function  such that  $(X,\p,a)$  admits the continuous linear extension $(\X,\tilde{\varphi},\tilde{a})$ (this holds, e.g., when $a\in C(\Delta,\B(F))$). 
Then
$$
\lambda(a,\p) =  \max_{\mu \in {\rm Inv}\left(\TDelta_\infty,\tilde{\varphi}\right)} \int_{\TDelta_\infty} 
\ln \tilde{a}(x, v) \,\, d\mu
= \ \max_{\mu \in {\rm Erg}\left(\TDelta_\infty,\tilde{\varphi}\right)}  \int_{\TDelta_\infty} 
\ln \tilde{a}(x, v) \,\, d\mu
$$
where $\TDelta_\infty$ is the essential domain of  $\tilde{\varphi}$. We assume here that $\ln (0) = -\infty$  and $\lambda(a,\p) =-\infty$ if  $ {\rm Erg}\left(\TDelta_\infty,\tilde{\varphi}\right) = \emptyset$  (this is the case when $\TDelta_\infty =\emptyset$ and all the more when $\Delta_\infty =\emptyset$).
\end{thm}
\begin{proof} Let $\TDelta_n$ be the domain of $\tilde{\varphi}^n$.
Simple calculation gives that 
\begin{equation}\label{eq:prod-coc}
\|C_{a^*,\p}^{b}(x,n)v\|=\prod_{k=0}^{n-1} 
\tilde{a} (\tilde{\varphi}^k (x,[v]))
\end{equation}
 whenever  $(x,[v])\in \TDelta_n$
and $\|C_{a^*,\p}^{b}(x,n)v\|=0$ otherwise. Moreover, by \eqref{eq:cocycles_vs_adjoints} we have $C_{a^*,\p}^{b}(x,n)=C_{a,\p}^{f}(x,n)^*$.
Thus we get
\begin{align*}
\lambda(a,\p)&=\lim_{n\to \infty}\sup_{x\in \Delta_n} \frac{1}{n} \ln \|C_{a,\p}^{f}(x,n)^*\|
=\lim_{n\to \infty}\sup_{x\in  \Delta_n}\sup_{v \in B} \frac{1}{n} \ln \|C_{a,\p}^f(x,n)^*v\|
\\
&=\lim_{n\to \infty}\sup_{(x,[v])\in  \Delta_n\times [B]} \frac{1}{n} \ln \|C_{a^*,\p}^b(x,n)v\|=  
\lim_{n\to \infty}  \sup_{(x,v)\in \TDelta_n} \frac{1}{n} \ln \left(\prod_{k=0}^{n-1} 
\tilde{a} (\tilde{\varphi}^k (x,v))\right)
\\
&= \lim_{n\to \infty}  \sup_{(x,v)\in \TDelta_n} S_n(\ln\tilde{a})(x,v) 
\end{align*}
where $S_n(\ln\tilde{a})$ is the empirical average  corresponding to $\ln\tilde{a}$ and $\tilde{\varphi}$, see \eqref{eq:ergodic_sum}.
Thus the assertion follows from Theorem \ref{Var_princ_for_ergodic_sums}.

\end{proof}
\begin{rem}
\label{relation-erg-Terg}
Note that   $\TDelta_\infty \subseteq \Delta_\infty\times [S]$ where $S$ is the unit sphere in $F^{*}$, and thus $[S]$ is the projective space of $F^*$.  Lemma~\ref{lem1.0} implies that
$$
{\rm Inv}\left(\TDelta_\infty,\tilde{\varphi}\right) = {\rm Inv}\left(\Delta_\infty \times [S],\tilde{\varphi}\right)
\quad\text{ and }\quad {\rm Erg}\left(\TDelta_\infty,\tilde{\varphi}\right) = {\rm Erg}\left(\Delta_\infty \times [S],\tilde{\varphi}\right).
$$ 
In particular, in  Theorem \ref{thm:varp-principle_for_cocycles} one could  replace $\TDelta_\infty$ with $\Delta_\infty\times [S]$.
Moreover, 
the projection $p:\Delta_\infty\times [S]\to \Delta_\infty$ onto the first coordinate induces a natural projection $p^*$ of measures:  for each Borel measure $\tilde{\mu}$ on $\Delta_{\infty}\times [S]$,  $p^*(\tilde{\mu})$ is a Borel measure  on $\Delta_{\infty}$ where
$
p^*(\tilde{\mu}) (\omega):= \tilde{\mu} (p^{-1}(\omega))$, $\omega \subseteq \Delta_\infty.
$
By definitions of $p^*$ and $\tilde{\p}$
it follows that
$$
p^*\left({\rm Inv}\left(\TDelta_\infty,\tilde{\varphi}\right)\right)\subseteq 
{\rm Inv}\left(\Delta_\infty,\p\right)
 \quad
\text{and}
\quad 
p^*\left({\rm Erg}\left(\TDelta_\infty,\tilde{\varphi}\right)\right)\subseteq 
{\rm Erg}\left(\Delta_\infty,{\varphi}\right).
$$
Moreover, if the weight $a:\Delta\to  \B(F)$  has some zeros we may replace $\Delta$ with 
$$\Delta^a:=\{x\in \Delta: a(x)\neq 0\},$$
which is an  open subset of $X$. Denoting by $\Delta_{\infty}^a$  the essential domain of  $\varphi:\Delta^a\to X$, we get
$p^*({\rm Inv}(\TDelta_\infty,\tilde{\varphi}))\subseteq 
{\rm Inv}(\Delta_\infty^a,\p)$ and 
$ 
p^*({\rm Erg}(\TDelta_\infty,\tilde{\varphi}))\subseteq 
{\rm Erg}(\Delta_\infty^a,{\varphi}).
$
Sometimes the system $(\Delta_\infty^a,{\varphi})$
might be much smaller  than $(\Delta_\infty,{\varphi})$, cf. Proposition \ref{prop:compact_spectral_radius} below.

\end{rem}
\begin{cor}\label{cor:for_spectral_exponent}
Retain the notation and assumptions of Theorem \ref{thm:varp-principle_for_cocycles}. There is $\tilde{\mu} \in {\rm Erg}\left(\TDelta_\infty,\tilde\p\right)$
and a subset $Y\subseteq \TDelta_\infty$, $\tilde{\mu} (Y)=1$, such that 
$$
\lambda(a,\p) =\lim_{n\to \infty}  \frac{1}{n}\ln \|C_{a^*,\p}^{b}(x,n)v\|\quad \text{ for every }y\in Y.
$$
\end{cor}
\begin{proof} In the last step in the proof of 
Theorem \ref{thm:varp-principle_for_cocycles} instead of Theorem \ref{Var_princ_for_ergodic_sums} apply Corollary \ref{cor:ergodic_sums}.  and use that $\prod_{k=0}^{n-1} 
\tilde{a} (\tilde{\varphi}^k (x,[v]))=\|C_{a^*,\p}^{b}(x,n)v\|$, cf. \eqref{eq:prod-coc}.
\end{proof}

We will define Lyapunov exponents with respect to a partial dynamical system $(X,\p)$ as the corresponding objects with respect to
the (full) dynamical system $(\Delta_\infty,\p)$. 
\begin{defn} Let  $C:\Delta_\infty\times \N \to \B(F)$ be a function.  \emph{Lyapunov exponent of  $C$ at a point $x\in \Delta_\infty$ in direction $v\in F$} is given by the formula
$$
\label{e-Lyp-exp-def}
\lambda_x(C,v):=\liminf_{n\to\infty}\frac{1}{n}\ln \|C(x,n)v\|.
$$
  \emph{The maximal Lyapunov exponent of  $C$ at $x$} is
$$
\lambda_x(C):=\liminf_{n\to\infty}\frac{1}{n}\ln \|C(x,n)\|.
$$
We also put 
$
\lambda(C):=\liminf_{n\to\infty}\sup_{x\in \Delta_\infty} \frac{1}{n}\ln \|C(x,n)\|
$.
\end{defn}
For any function $C:\Delta_\infty\times \N \to \B(F)$, for which the corresponding Lyapunov exponents exist, we clearly have 
$$
\sup_{x\in \Delta_\infty}\sup_{v\in F, \|v\|=1} \lambda_x(C,v) \leq  \sup_{x\in \Delta_\infty} \lambda_x(C)\leq \lambda(C).
$$
Note that the left most expression involves the least number of conditions, and hence is the easiest to calculate. 
The chief importance   of backward cocycles associated with operator valued functions lies in that for such cocycles the three above expressions are equal. Moreover, exploiting the ergodic properties of the considered systems we may substantially decrease the domains of the suprema:
\begin{thm}[Variational principle using Lyapunov exponents I]
 \label{thm:variational_principle_Lyapunov} Let $(X,\p)$ be a partial dynamical system and  let $\Delta\ni x \to a(x) \in \B(F)$ a bounded function  such that  $(X,\p,a)$  admits the continuous linear extension  (this holds, e.g., when $a\in C(\Delta,\B(F))$).
Let $\Omega\subseteq \Delta_\infty$ be any set  such that $\mu(\Omega)>0$ for every $\mu \in {\rm Erg} (\Delta_\infty,\p)$. Then
\begin{equation}
\label{eq:Lyapunov_essential}
\lambda(a,\p) =\lambda(C_{a,\p}^f )= \max_{x \in \Omega} \lambda_x(C_{a,\p}^f)=\max_{(x,v) \in \Omega\times S}\lambda_x(C_{a^*,\p}^b,v) 
\end{equation}
where $S$ is the unit sphere in $F^*$. The set $\Delta_\infty$ above can be replaced by the essential domain
$\Delta^a_{\infty}$ 
of  $\varphi:\Delta^a=\{x\in \Delta: a(x)\neq 0\}\to X$.
\end{thm}
\begin{proof} By \eqref{eq:cocycles_vs_adjoints}  and the  definition of $\lambda(a,\p)$ for any $(x,v) \in \Delta_{\infty}\times S$   we have 
$$
\lambda_x(C_{a^*,\p}^{b},v)\leq \lambda_x(C_{a^*,\p}^{b})=\lambda_x(C_{a,\p}^{f})\leq \lambda(C_{a,\p}^f ) \leq \lambda(a,\p).
$$
Thus it suffices to find $(x,v) \in \Omega\times S$ with $\lambda(a,\p)=\lambda_x(C_{a^*,\p}^{b},v)$. To this end, take $\tilde{\mu}\in {\rm Erg} (\TDelta_\infty,\tilde{\p})$ and $Y\subseteq  \TDelta_{\infty}\subseteq \Delta_{\infty}\times [S]$ as in  Corollary \ref{cor:for_spectral_exponent}. 
That is, we have $\tilde{\mu}(Y)=1$ and  $\lambda(a,\p)=\lambda_x(C_{a^*,\p}^{b},v)$ for every $(x,[v])\in Y$. 
Using the projection $p:\Delta_\infty\times [S]\supseteq \TDelta_\infty\to \Delta_\infty$, 
we get $\mu := p^*(\tilde{\mu})\in {\rm Erg} (\Delta_\infty,\p)$, cf. Remark \ref{relation-erg-Terg}.
Since $\mu(p(Y))=1$ we have $p(Y)\cap \Omega \neq \emptyset$. For every $x\in p(Y)\cap \Omega$ there is $v\in S$ such that $(x,[v])\in Y$ and therefore $\lambda(a,\p)=\lambda_x(C_{a^*,\p}^{b},v)$.
\end{proof}
\begin{rem} In the above assertion one can take $\Omega$ to be the  set of non-wandering points $\Omega(\varphi)$ for the  partial map $\varphi:\Delta \to X$ 
or even better the corresponding set $\Omega(\varphi|_{\Delta^a})$ for the restricted partial map $\varphi:\Delta\supseteq \Delta^a \to X$. 
\end{rem}
We may rephrase the above theorem in more appealing (but slightly weaker) form, using measure exponents:

 \begin{lem}\label{prop:Lyapunov_from_measures}
Let  $a$ be such that $\Delta_n\ni x\to \|C_{a,\p}^{f}(x,n)\|$ is bounded  and measurable,
for every $n\in \N$. For any $\mu\in {\rm Erg} (\Delta_\infty,\p)$ there exists a number  $\lambda_\mu(a,\p)\in [-\infty, +\infty)$ such that
$$
 \lambda_\mu({a,\p})=\lambda_x(C_{a,\p}^f)=\lambda_x(C_{a^*,\p}^b) \qquad \text{for $\mu$-almost every }x\in \Delta_\infty.
$$
\end{lem}
\begin{proof}
The sequence of functions
$
F_n(x):=\ln \|C_{a,\p}^{f}(x,n)\|=\ln \|C_{a^*,\p}^{b}(x,n)\|,
$
is subadditive in the sense that  $F_{m+n}\le F_{m}\circ \p^n+ F_{n}$. Thus the assertion follows from Kingman's subadditive theorem, see for instance  \cite[Theorem~A.1]{Ruelle}.
\end{proof}
\begin{defn}\label{defn:Lyapunov_measures}
We call the number $\lambda_\mu({a,\p})$ defined in Lemma \ref{prop:Lyapunov_from_measures} the \emph{(maximal) measure exponent}
of $a$ with respect to the ergodic (partial) system $(X,\p,\mu)$. 
\end{defn}
\begin{cor}[Variational principle using Lyapunov exponents II]
 \label{cor:variational_principle_Lyapunov} Let $(X,\p)$ be a partial dynamical system and  let $\Delta\ni x \to a(x) \in \B(F)$ a bounded function  such that  $(X,\p,a)$  admits the continuous linear extension  (this holds, e.g., when $a\in C(\Delta,\B(F))$). Then
$$
\lambda(a,\p) =\max_{\mu \in {\rm Erg} (\Delta_\infty,\p)}\, \lambda_\mu({a,\p}).
$$
That is, the spectral exponent is the maximum of measure exponents. The set $\Delta_\infty$ above can be replaced by the essential domain
of  $\varphi:\{x\in \Delta: a(x)\neq 0\}\to X$. 
\end{cor}
\begin{proof} It follows from Theorem \ref{thm:variational_principle_Lyapunov} along with Lemma~\ref{prop:Lyapunov_from_measures}.
\end{proof}
\begin{rem}\label{rem:MET} The above variational principles are closely related with   Multiplicative Ergodic Theorems (MET). 
The first MET  was established by Oseledets \cite{oseled} and various generalizations keep appearing until the present times, see \cite{GTQ} and references therein. 
All MET's apply to backwards cocycles and require some compacntess assumptions. In particular, if one wants to combine them with the formula for the spectral exponent $\lambda(a,\p)$ one needs to pass to duals.
For our purposes versions of  MET's in \cite{Thieullen} and \cite{GTQ}    seem the most relevant. They imply  that if $\Delta\ni x \to a^*(x) \in \B(F^*)$ is asymptotically compact, and either $F^{**}$ is separable (\cite[Theorem 14]{GTQ})
or $a^*$ is continuous and $\Delta_\infty$ is a complete metric space (\cite[Theorem 2.3]{Thieullen}), 
then for every $\mu \in {\rm Erg} (\Delta_\infty,\p)$ we have Oseledets filtrations  of $F^*$ for the cocycle $C_{a^*,\p}^b$.
That is, for every $\mu \in {\rm Erg} (\Delta_\infty,\p)$   there exists  a  decrising sequence of numbers $\{\lambda_{i}\}_{i=1}^r$, where $1\leq r \leq \infty$,
 such that for $\mu$-almost every $x\in \Delta_\infty$  there is a measurable filtration of closed subspaces, 
$F^* = V_1(x) \supsetneq  V_2(x) \supsetneq \dots   \supseteq V_{r+1}(x):=\bigcap_{i=1}^r V_{i}(x)$ satisfying
 $a^*(x)(V_i(x))\subseteq V_i(\p(x))$,   the codimension of  $V_i(x)$ is finite and does not depend on $x$, 
$\lambda_x(C_{a^*,\p}^{b},v)=-\infty$  for every $v\in V_{r+1}(x)$, and  
$$
\lambda_i=\lambda_x(C_{a^*,\p}^{b},v) \quad \text{ for every $v\in V_i(x)\setminus V_{i+1}(x)$ and each $i<r+1$}.
$$
In particular, $\lambda_{\mu}(a)=\lambda_1$ realizes as a Lyapunov exponent $\lambda_x(C_{a^*,\p}^{b},v)$  in the direction $v$ for every
$v\in V_1(x)\setminus V_{2}(x)$ and $\mu$-almost every $x\in \Delta_\infty$. 
We stress that  by Theorem \ref{thm:variational_principle_Lyapunov}  the spectral exponent $\lambda(a,\p)$ always realizes as  a Lyapunov exponent $\lambda_x(C_{a^*,\p}^{b},v)$  for some direction $v$
and some measure $\mu$, without any compactness assumption on $a$!
\end{rem}

\section{Spectral radius of weighted endomorphisms of $A=C(X,D)$} 
\label{sec:spectral_radius}

 In this section, we use VPs obtained in the previous section, to give formulae for the spectral radii $r(a\alpha)$, $a\in A$, under the following assumptions:
\begin{itemize}
\item $A\cong C(X,D)$ where $D$ is a unial Banach algebra;
\item $\alpha: A\to A$ is contractive and $\alpha(C(X)\otimes 1)\subseteq \alpha(1)C(X)\otimes 1$. Then  by Proposition \ref{tensor podstawowe},  $\alpha$ generates a partial dynamical system $(X,\p)$, see Definition \ref{def:endomorphism_generate_partial_dynamical_system}. 

\end{itemize}
As we have seen in Proposition \ref{first radius}, the spectral radius of an abstract weighted shift operator $aT:E\rightarrow E$ coincides with the spectral radius  of the associated weighted endomorphism
$a\alpha:A \rightarrow A$. 
 In particular,  the situation where $E=L^p(X,\B(F))$ is a vector-valued function space, $T$ is a composition operator, and $A$ consists of operators of multiplication by operator-valued functions in $C(X,D)$, where  $D\subseteq \B(F)$.
But the developed  formulae can also  be   applied when $T$ is not a priori a composition operator. 

We start with two extremal cases, when \(D=\C\) is trivial or \(X=\{x\}\) is trivial.  Then we  get, respectively: the formula for  $r(a\alpha)$ where $A$ is a
commutative uniform algebra (subsection \ref{s-comm}), and    a 'Dynamical Variational Principle' and an intriguing version of Gelfand's formula  for the spectral radius $r(a)$ of an arbitrary operator $a\in B(F)$ (subsection \ref{DVP-arb}).

In subsection \ref{s-varp-space} we consider the case where $D=\B(F)$ and the associated field of endomorphisms $\{\alpha_x\}_{x\in \Delta}$ consists of  inner endomorphism of $\B(F)$. 
Finally, in subsection \ref{varp-gen} we derive formulas for spectral radius where \(D\) is arbitrary. We achieve this by reducing the general case to the special one treated in subsection \ref{s-varp-space}.

\subsection{Variational principle for commutative algebra of weights} 
\label{s-comm}
Here  we assume that $A$ is a \emph{uniform algebra}, sometimes also called function algebra \cite[30.1]{Zelazko}. 
Thus $A$ is a commutative Banacha algebra such that $\|a\|=r(a)$ for every  $a\in A$.
Equivalently, this means that  the Gelfand transform on $A$ is an isometry, and   we may view $A$ as a closed subalgebra of  $C(X)$. 
 The following variational principle   generalizes and unifies  the corresponding results in  \cite{Kitover}, \cite{Lebedev79}, \cite[4]{Anton_Lebed}, \cite[5]{Anton}, \cite{kwa-phd}. We could deduce it either from  one of Theorems \ref{thm:varp-principle_for_cocycles}, \ref{thm:variational_principle_Lyapunov} or from the formula for spectral radius of a weighted unital endomorphism on $C(X)$ \cite[5]{Anton}. 
We give a short proof based on Theorem~\ref{Var_princ_for_ergodic_sums}: 
 \begin{thm}\label{promień spektralny dla sumy}
Let $\alpha:A\to A$ be an endomorphism  of a uniform algebra $A$. Let  $(X,\varphi)$ be the partial dynamical system dual to $(A,\alpha)$, see Definition \ref{odwzorowania czesciowe label}. Then for any  $a \in A$ the spectral radius  of the weighted endomorphism $a\alpha:A\to A$ is given by 
\begin{equation}\label{wzór na promień spektralny aT}
r(a\alpha)= \max_{\mu\in {\rm Inv} (\Delta_\infty,\p)} \, \exp \int_{\Delta_\infty} \ln|\widehat{a}(x)|\, d\mu = \max_{\mu\in {\rm Erg} (\Delta_\infty,\p)} \, \exp \int_{\Delta_\infty} \ln|\widehat{a}(x)|\, d\mu,
\end{equation}
where $\ln (0) = -\infty$,  $\exp(-\infty)=0$ and  $r(a\alpha)=0$ if  $\Delta_\infty = \emptyset$.  For a fixed $a\in A$, 
the set $\Delta_\infty$ above can be replaced by the essential domain
of  $\varphi:\{x\in \Delta: a(x)\neq 0\}\to X$. 
\end{thm}
\begin{proof}
We have $r(a\alpha)=\lim_{n\to \infty}\| a\alpha(a) {\dots}\, \alpha^{n}(a)\|^{\frac{1}{n}}$ by Proposition \ref{first radius}.  
Using this and that the Gelfand transform is isometric on $A$ we get 
\begin{align*}
\ln r(a\alpha) &= \lim_{n\to \infty}{\frac{1}{n}} \ln\| a\alpha(a) {\dots}\, \alpha^{n}(a)\| =\lim_{n\to \infty}\sup_{x\in \Delta_n} \frac{1}{n}\sum_{k=o}^{n-1} \ln |\widehat{\alpha^k(a)}|(x) 
\\
&=\lim_{n\to \infty}\sup_{x\in \Delta_n} S_n (\ln |\widehat{a}|)(x).
\end{align*}
Thus the assertion follows from Theorem~\ref{Var_princ_for_ergodic_sums}.
	\end{proof}

\begin{rem}\label{uwagi o zasadzie wariacyjnej}
  Formula  \eqref{wzór na promień spektralny aT} can be  improved in the following sense, cf. \cite{Kitover}.  A closed subset $F\subseteq X$ is called a \emph{maximizing  set} for algebra  $A$ if $\|a\|=\max_{x\in F}|\widehat{a}(x)|$ for every  $a\in A$. There exists a uniquely defined minimal maximizing set for  $A$ which is called \emph{Shilov boundary} and is denoted by $\partial A$. If a map  $\p$ preserves  $\partial A$, which is always the case when   $\alpha:A\to A$ is an epimorphism (one can apply  \cite[Theorem~15.3]{Zelazko}), then  for every $a\in A$,  $\p$ preserves  $\partial A \cap \Delta_{\infty}^a$ where $\Delta_{\infty}^a$ is the essential domain for 
	$\varphi:\{x\in \Delta: a(x)\neq 0\}\to X$ and 
	$$
r(a\alpha)=\max_{\mu\in {\rm Erg} (\partial A \cap \Delta_{\infty}^a,\p)}\, \exp \int_{\partial A \cap \Delta_\infty^a} \ln|\widehat{a}(x)|\, d\mu\,.
$$ 
\end{rem}
\begin{ex}[Classical weighted shift operators] 
 \label{promienisty operator klasyczny}
 In this example we  show the relation between the above considered variational principle  (Theorem~\ref{promień spektralny dla sumy}), Banach limits and the known formula for the spectral radius of the classical weighted shift operator. 
Also  $\lim\sup$ variational principle  (Theorem~\ref{Var_princ_for_ergodic_sums})  arises  herewith naturally.
 Namely, let $aT_\N$ be the classical weighted shift operator with a weight $a\in  \ell^\infty (\N)$ 
acting in one of the spaces $\ell^p(\N)$, $ p\in [1, \infty]$, $c(\N)$ or $c_0(\N)$, cf. Example~\ref{operator wazonego przesuniecia na l_p}. 
In this case, the spectrum $\beta(\N)$ of  $\ell^\infty (\N)$ is the Stone-\v{C}ech compactification of 
$\N$, and $T_\N$ generates on $\beta(\N)$  the map $\varphi$ such that $\varphi(n)=n+1$ for $n\in \N \subseteq \beta(\N)$.  
In particular, the set $\Delta_\infty=\bigcap_{n\in \N}\varphi^n(\beta(\N))$ is  the corona $\beta(\N)\setminus \N$ 
and every $\varphi$-invariant measure is contained in $\beta(\N)\setminus \N$. 
In other words, the functional $\phi_\mu \in \big(C(\beta(\N)\big)^*$ corresponding 
to a measure $\mu \in {\rm Inv}(\beta(\N),\varphi)$ 
satisfies conditions 
$$
\|\phi_\mu\|=\phi_\mu(1)= 1, \qquad  \phi_\mu (a\circ \varphi)=\phi_\mu(a).
$$
Hence it is a \emph{Banach limit} on $\ell^\infty_\R(\N)$. Moreover, each Banach limit is of this form. Therefore, denoting by $\text{\bf BL}$ 
the set of all Banach limits,  the variational principle in Theorem~\ref{promień spektralny dla sumy}, formula \eqref{wzór na promień spektralny aT}, 
in this case assumes the form: 
\begin{equation}\label{Banach limity na jutro}
r(aT_\N)=\max_{\phi\in \text{\bf BL}} \, \exp \phi\left(\{ \ln|a(n)|\}_{n\in \N}\right),
\end{equation}
where  $\phi (\ln |a|)$ is treated as the integral with respect to the corresponding measure, when  $\ln |a|$ does not belong to $\ell^\infty_\R(\N)$.
This formula is the most effective for the so-called \emph{almost convergent sequences}, that is elements $a\in \ell^\infty_\R(\N)$ for which the value $\phi(a)$ does not depend on the choice of a Banach limit $\phi\in \text{\bf BL}$. As examples of sequences of this sort one can take convergent sequences, periodic sequences and their sums. 
In general, as shown by G. G. Lorentz \cite{Lor}, a sequence $a=\{a(n)\}_{n\in\N}\in \ell^\infty_\R(\N)$ is almost convergent iff the sequence $\{a_n\}_{n\in \N}
\subseteq \ell^\infty_\R(\N)$ of means  $a_n(k):=\frac{1}{n}\sum_{i=0}^{n-1}a(k+i)$ is convergent in $\ell^\infty_\R(\N)$ to a constant sequence and in this case the value of this constant is equal to the Banach limit for each Banach limit of $a$. This statement readily follows from the formula
$$
\max_{\phi\in \text{\bf BL}}\phi(a)=  \,\lim_{n\to \infty} \sup_{k\in \N}\frac{1}{n}\sum_{i=0}^{n-1}a(k+i),
$$
which in turn follows from our $\lim\sup$ variational principle (Theorem~\ref{Var_princ_for_ergodic_sums}).
Thus our results recover Lorentz's theorem. Combining the above expression with  \eqref{Banach limity na jutro} 
we get  
$$
r(aT_\N)=\, \lim_{n\to \infty} \sup_{k\in \N} \sqrt[n]{\prod_{i=0}^{n-1}|a(k+i)|},
$$
which is the standard formula (see  \cite[Problem 77]{Halmos}). Formula \eqref{Banach limity na jutro}  seem more abstract but also more 
efficient. For instance, if  $a\in \ell^\infty (\N)$ is such that the sequence  $\{ \ln|a(n)|\}_{n}$ is almost convergent, then in   \eqref{Banach limity na jutro} it is enough to apply any Banach limit, for example, taking  Ces\`aro mean one obtains  
$
r(aT_\N)=\, \lim_{n\to \infty} \sqrt[n]{\prod_{k=0}^{n-1}|a(k)|}.
$  
\end{ex}

\subsection{Dynamical variational principle for an arbitrary operator}\label{DVP-arb}
Let $a\in \B(F)$ be \emph{an arbitrary} operator acting on an arbitrary Banach space $F$. 
Let  $[S]=S/\sim$ be the projective space associated to the dual space $F^*$, cf. Remark \ref{relation-erg-Terg}.
 Consider the partial mapping $\tilde{\varphi} : \TDelta \to  {[S]}$  given by
\begin{equation}
\label{arb-VP-phi}
\tilde{\varphi} ([v]) := 
\left[\frac{a^*(v)}{\Vert a^*(v)\Vert}\right], \qquad  \TDelta:= \{ [v]: a^*(v) \neq 0\}.
\end{equation}
\begin{thm}
\label{arb-VP}
Let $F$ be a Banach space. Then for every $a\in \B(F)$
\begin{equation}
\label{arb-VP-form}
r(a) = \max_{\mu \in {\rm Inv}({[S]}\,, \tilde{\varphi})} \exp \int_{{[S]}} \ln \Vert a^* (v)\Vert \, d\mu =
\max_{\mu \in {\rm Erg}({[S]}\,, \tilde{\varphi})} \exp \int_{{[S]}} \ln \Vert a^* (v)\Vert \, d\mu
\end{equation}
where $\tilde{\varphi}$ is given by \eqref{arb-VP-phi},  and $r(a) =0$ if $ {\rm Erg}\left({[S]}\,,\tilde{\varphi}\right) = \emptyset$.  
\end{thm}
\begin{proof} Note that $\ln r(a)=\lambda(a,\p)$ where  $X:= \{x\}$ is a singleton and $\varphi (x) := x$ is the identity map on $X$, cf. Remark \ref{rem:spectral_exponents_and_radii}. 
Hence the assertion follows from  Theorem ~\ref{thm:varp-principle_for_cocycles} and Remark \ref{relation-erg-Terg}.
\end{proof}
Obviously, one should not expect that the formula \eqref{arb-VP-form} simplifies calculation of the spectral radius of $a$.
Nevertheless, it allows intriguing dynamical interpretations and has  interesting consequences:
\begin{ex}\label{ex:2_times_2_matrix} 
Let us consider an invertible operator acting on a two-dimensional complex space. So that we may identify it with a 
$2 \rtimes 2$ matrix 
$\mathbb{A}=\left(\begin{array}{cc}
a & b 
\\
c & d
\end{array}\right)
$. The corresponding projective space $[S]=\C P^1$ can be identified with the Riemann sphere $\overline{\C}$,
and then $\tilde{\varphi}$ becomes the M\"obius transformation $\tilde{\varphi}(z)=\frac{az + b}{cz +d}$
and \eqref{arb-VP-form} assumes the form
$$
r(\mathbb{A})=\max_{\mu \in {\rm Erg}(\overline{\C}, \tilde{\varphi})} \exp \int_{\overline{\C}} \ln f_{\mathbb{A}}(z)\, d\mu,
$$
where $ f_{\mathbb{A}}(z):=(|az+b|+|cz+d|)/(|z|+1)$, for $z\in \C$, and $ f_{\mathbb{A}}(\infty):=|a|+|c|$. 
One may verify that $\tilde{\varphi}$ either has exactly one attracting point $z_0$, or it has at least 
two neutral fixed points. In the first case $\delta_{z_0}$ is the unique ergodic measure
for which the above maximum is attained, and so  $r(\mathbb{A})= f_{\mathbb{A}}(z_0)$. In the second case
the maximum is attained in  $\delta_z$ for every neutral fixed point $z\in \overline{\C}$ and so
$ r(\mathbb{A})= f_{\mathbb{A}}(z)$ for any such point. In both cases  M\"obius transformation has topological entropy zero.
Thus  the spectral 
exponent $\ln r(\mathbb{A})$ is equal to the topological pressure $P(\tilde{\varphi}, \ln f_{\mathbb{A}})$
of $\tilde{\varphi}$ with potential $\ln f_{\mathbb{A}}$.




\end{ex}
We obtained Theorem \ref{arb-VP} as a special case of Theorem ~\ref{thm:varp-principle_for_cocycles}. 
Applying in the same manner Corollary \ref{cor:variational_principle_Lyapunov} one gets nothing but 
$r(a)=\lim_{n\to \infty}\| a^n\|^{\frac{1}{n}}$. 
However, applying formula \eqref{eq:Lyapunov_essential} in Theorem \ref{thm:variational_principle_Lyapunov}
 gives  an interesting
 improvement of the Gelfand's formula:
\begin{cor}[refined Gelfand's formula]
\label{cor-Gelf}
For every operator $a\in \B(F)$ on a Banach space $F$ we have 
$$
r(a)= \max_{v\in S} \lim_{n\to \infty} \|a^{*n}v\|^{\frac{1}{n}}
$$
where $S$ is the unit sphere in the dual space $F^*$. Thus if $F$ is a reflexive  space, 
 the spectral exponent  is always the  Lyapunov exponent in some direction $v\in F$:
$$
r(a)= \lim_{n\to \infty} \|a^{n}v\|^{\frac{1}{n}}.
$$ 
\end{cor}
\begin{proof}
The first part follows from Theorem \ref{thm:variational_principle_Lyapunov} applied to the case  $X:= \{x\}$ and $\varphi (x) := x$. 
This implies that  $r(a)=r(a^*)=\lim_{n\to \infty} \|(a^{**n})v\|^{\frac{1}{n}}$, for some $v\in F^{**}$. 
And if $F$ is reflexive we may assume identifications  $F\cong F^{**} $ and $a = a^{**}$.  
\end{proof}

\subsection{Variational principles  in the case $A=C(X,\B(F))$}
\label{s-varp-space}
In this subsection we make the following standing assumptions 
\begin{itemize}
\item $A=C(X,\B(F))$ where $F$ is a Banach space;
\item $\alpha:A\to A$ is an endomorphism 
such that the corresponding field of  (non-zero) endomorphisms $\{\alpha_x\}_{x\in \Delta}$ of
 $\B(F)$  consists of inner isometric monomorphisms. 
\end{itemize}
Then by Propositions \ref{zapowiedziane twierdzenie}, \ref{prop:on_inner_endo} we  have
\begin{equation}\label{eq:desired_endomorphism_form}
\alpha(a)(x)=
\begin{cases}
T_x a(\p(x))S_x,& x \in \Delta,\\
0,& x\notin \Delta,
\end{cases}
 \qquad  a \in C(X,\B(F)),
\end{equation}
where $\p:\Delta \rightarrow X$ is a continuous partial map and  $\{T_x, S_x\}_{x\in \Delta}\subseteq \B(F)$ are such that $S_xT_x=1$ and $T_x$ is an isometry, for every $x\in \Delta$.
In this situation we may extend Remark \ref{rem:spectral_exponents_and_radii} and express the spectral radius of \(a\alpha\) by a spectral exponent of a forward cocycle associated with   $\Delta\ni x \to a(x)T_x$:
\begin{prop}
\label{prop:coc-aT} With the above assumptions, for every  $a\in A=C(X,\B(F))$ we have
$$
r(a\alpha) 
= \lim_{n\to \infty} \sup_{x\in \Delta_n}\Vert C_{aT, \varphi}^f (x,n)\Vert^{\frac{1}{n}}= e^{\lambda(aT,\varphi)},
$$
where $aT$ denotes the function $\Delta\ni x \to a(x)T_x\in \B(F)$, and 
$C_{aT,\p}^f(x,n):=a(x)T_x\cdot a(\p(x))T_{\p(x) } \cdot{\dots} \cdot  
a(\p^{n-1}(x))T_{\p^{n-1}(x)}$ is the associated forward cocycle.
\end{prop}
\begin{proof} We introduce the following notation: for any sequence of points $x_1,{\dots},x_n\in X$ and
any field of operators $X\ni x\to R_x\in \B(F)$  we put
$
R_{x_1x_2{\dots}x_n}:=R_{x_1}R_{x_2}{\dots}R_{x_n}.
$
Thus 
for every $k \in \N$,  $a \in C(X, \B (F))$ and $x\in \Delta_k$ we have
$$
\alpha^k (a) (x) = T_{x\varphi (x)\dots\varphi^{k-1} (x)}\, a(\varphi^k (x))\, S_{\varphi^{k-1} (x)\dots\varphi (x)x}.
$$
Note that $S_xT_x =1$ for $x\in \Delta$ implies
 $S_{\varphi^{n} (x)\dots \varphi (x)x} T_{x\varphi (x)\dots\varphi^{n} (x)}=1$ for $x\in \Delta_{n+1}$. 
Thus putting 
$C_{a,\alpha}(x,n):=a\cdot \alpha(a)\cdot {\dots} \cdot \alpha^{n-1} (a) (x)$  for   $x\in \Delta_{n}$, 
cf. Remark \ref{rem:spectral_exponents_and_radii},  for  every $x\in \Delta_{n+1}$ we get 
$$
C_{a,\alpha}(x,n+1)=
C_{aT, \varphi} (x,n) a(\varphi^{n} (x))  S_{\varphi^{n-1} (x)\dots \varphi (x)x}.
$$
As $\{S_x, T_x\}_{x\in \Delta}$, are contractions it follows that
$$
\Vert C_{a,\alpha}(x,n+1) \Vert
\le \Vert a\Vert\cdot
\Vert C_{aT, \varphi} (x,n) \Vert.
$$
On the other hand, we have 
\begin{align*}
\Vert C_{aT, \varphi} (x,n+1) \Vert 
&= \Vert  C_{aT, \varphi} (x,n)  
a(\varphi^{n} (x))T_{x\varphi (x)\dots\varphi^{n-1} (x)}\Vert \leq  \Vert  C_{aT, \varphi} (x,n)  
a(\varphi^{n} (x))\Vert
\\
&=\Vert  C_{aT, \varphi} (x,n)  
a(\varphi^{n} (x))   S_{\varphi^{n-1} (x)\dots \varphi (x)x} T_{x\varphi (x)\dots\varphi^{n-1} (x)}\Vert 
\\
&
\leq \Vert  C_{aT, \varphi} (x,n)  
a(\varphi^{n} (x))   S_{\varphi^{n-1} (x)\dots \varphi (x)x}\Vert =\Vert C_{a, \al} (x,n+1) \Vert .
\end{align*}
Combining the above inequalities we get 
\begin{equation}
\label{e-coc-coc1}
\Vert C_{a,\alpha}(x,n+1) \Vert
\le \Vert a\Vert\cdot
\Vert C_{aT, \varphi} (x,n) \Vert  \le \Vert a\Vert  \cdot
\Vert C_{a,\alpha}(x,n)\Vert.
\end{equation}
However,  $r(a\alpha) = \lim_{n\to \infty} \max_{x\in \Delta_n}\Vert C_{a, \alpha} (x,n)\Vert^{\frac{1}{n}}$ 
by Corollary \ref{rem:spectral_exponents_and_radii}.
Thus  inequalities \eqref{e-coc-coc1} give $
r(a\alpha) 
= \lim_{n\to \infty} \sup_{x\in \Delta_n}\Vert C_{aT, \varphi} (x,n)\Vert^{\frac{1}{n}}
$.
\end{proof} 
In general the map $x \to aT(x)=a(x)T_x$ may be far from being continuous, cf.  Example \ref{ex-unmeas-1}  and Remark \ref{rem:cohomological remark}. 
Nevertheless, as we show now,  the triple $(X,\p,aT)$ admits the continuous linear extension in the sense of Definition  \ref{defn:linear_extension}. To this end recall that  $B$ is the unit ball in the dual space $F^*$ equipped with $^*$-weak topology, and $[B]$ is the factor space $B/\sim$ where $[v]=\{w\in B: |v|=|w|\}$ is the equivalence class of $v\in B$.

\begin{lem}\label{lem:continuous_map_into_projective_ball}
For every $a\in A$ and $v\in F^*$,  the map $\Delta\ni x \mapsto [T_x^* a(x)^* v ]\in [B]$ is continuous.
\end{lem}
\begin{proof}
The topology in $[B]$ is generated by sets 
$$
U_{w,h,\varepsilon}:=\{[u]\in  [B]: \exists_{\lambda \in \T} \,\, |w(h)-\lambda u(h)|<\varepsilon\} 
$$
where  $w\in F^*$, $h\in F$ and $\varepsilon >0$. 

Let us take any $x_0\in\Delta$ and suppose that $[T_{x_0}^*a(x_0)^* v ]\in U_{w,h,\varepsilon}$ for some $w,h,\varepsilon$. That is, there is  $\lambda\in \T$ such that  
$$
| w(h)-\lambda v(a(x_0)T_{x_0}h)|<\varepsilon.
$$
 Let $\delta_1, \delta_2>0$ be arbitrary. Let $V\subseteq \Delta$  be an open neighbourhood of $x_0$ such that 
$\|a(x)-a(x_0)\|< \delta_1$.
 By Lemma \ref{lem:on_choice_of_S_x_T_x}(2) we may find an open  neighbourhood $U$ of $x_0$ contained in $V$   
and numbers $\{\lambda_x\}_{x\in U}\subseteq \T$ such that 
$\| \lambda_xT_xh -T_{x_0}h\|< \delta_2$. Then for every $x\in U$ we have  
\begin{align*}
| w(h)-\lambda \lambda_x v(a(x)T_{x}h)|&\leq | w(h)- \lambda v(a(x_0)T_{x_0}h)|
+ |\lambda v(a(x_0)T_{x_0}h) - \lambda  v(a(x)T_{x_0}h)|  
\\
&\,\,\,\,\,\,\,\, + |\lambda  v(a(x)T_{x_0}h) - \lambda \lambda_x v(a(x)T_{x}h)|
\\
& \leq  | w(h)- \lambda v(a(x_0)T_{x_0}h)| + \delta_1 \|v\| \|T_{x_0}h\| + \delta_2 \|v\| \|a(x)\| 
\\
&< | w(h)- \lambda v(a(x_0)T_{x_0}h)| + \delta_1 \|v\| \|T_{x_0}h\| + \delta_2 \|v\| (\|a(x_0)\|+\delta_1).
\end{align*}
Clearly, we may assume that $v\neq 0$. Then we put 
$$
\delta_2:=\frac{\varepsilon - | w(h)-\lambda  v(a(x_0)T_{x_0}h|}{2  \|v\| (\|a(x_0)\|+\delta_1)}.
$$
If $\|T_{x_0}h\|=0$ this already gives $| w(h)-\lambda \lambda_x v(a(x)T_{x}h)|< \varepsilon$.
If $\|T_{x_0}h\|\neq 0$, then putting 
$$
\delta_1:=\frac{\varepsilon - | w(h)-\lambda  v(a(x_0)T_{x_0}h|}{2  \|v\| \|T_{x_0}h\|}
$$
 we also get $| w(h)-\lambda \lambda_x v(a(x)T_{x}h)|< \varepsilon$.
Thus for every $x\in U$ we get
$[T_{x}^*a(x)^* v ]\in U_{w,h,\varepsilon}$. 
 This gives the assertion.
\end{proof}

\begin{lem}\label{lem:continuous_map_in_norm}
For every $a\in A$ and $v\in F^*$,  the map $\Delta\ni x \mapsto \|T_x^* a(x)^* v \|\in [0,\infty)$ is continuous.
\end{lem}
\begin{proof}
Recall that $T_x$ is an isometry from $F$ onto $\alpha_x(1)F$ and $\alpha_x(1)$ is a norm one projection. Hence 
$$
 \|T_x^* a(x)^* v \|=\sup_{\|h\|=1} \|v\big(a(x)T_x h\big) \|
=\sup_{\|h\|=1} \|v\big(a(x)\alpha_x(1) h\big) \|=\|\big(a(x)\alpha_x(1)\big)^* v \|.
$$
Thus the assertion follows from the continuity of  
$\Delta\ni x \mapsto a(x)\alpha_x(1)\in \B(F)$.
\end{proof}

\begin{thm}
\label{thm:varp-principle-TS} 
Let $A = C(X, \B(F))$ where $F$ is a Banach space  and suppose that   $\alpha : A\to A$ is an endomorphism of $A$ such that the generated field of endomorphisms $\{\alpha_x\}_{x\in \Delta}$ consists of isometric inner endomorphism of $\B(F)$. Then there is a dual partial dynamical system $(X,\p)$ and a family of isometries  $\{T_x\}_{x\in \Delta}\subseteq \B(F)$ satisfying 
$T_xb=\alpha_x(b)T_x$ for $b\in \B(F)$. For any $a\in A$ the map $\Delta\ni x \to aT(x):=a(x)T_x$ admits the linear extension  in the sense of Definition  \ref{defn:linear_extension}. In particular, fixing \(a\in A\) and defining
\begin{enumerate}
\item $\X:=X\times [S]$ where $[S]$ is a projective space of the dual Banach space $F^*$,
\item $
\TDelta:= \{(x,[v])\in \widetilde{X}: \, x\in \Delta,\,\, T^{*}_x a^{*}(x)v  \neq 0\},
$ and 
$$
\tilde{\p} (x,[v]):= \left(\varphi (x), \left[\frac{T^*_{x}a^*(x)v}{\| T^*_{x}a^*(x)v\|}\right]\right), \quad \tilde{a}(x,[v]):= 
\| T^*_{x}a^*(x)v\|, \quad  (x,v) \in \TDelta,
$$
\end{enumerate}
we get
$$
r(a\alpha)  = \max_{\mu \in {\rm Inv}\left(\TDelta_\infty,\tilde{\varphi}\right)} \, \exp \int_{\TDelta_\infty} 
\ln \tilde{a}(x, [v]) \,\, d\mu = \ \max_{\mu \in {\rm Erg}\left(\TDelta_\infty,\tilde{\varphi}\right)} \, \exp \int_{\TDelta_\infty} 
\ln \tilde{a}(x, [v]) \,\, d\mu
$$
where $\TDelta_\infty$ is the essential domain of $\tilde{\varphi}$.
\end{thm} 
\begin{proof} Let us fix $a\in A=C(X,\B(F))$ and define the triple $(\X,\tilde{\varphi}, \tilde{a})$ as in the assertion, but with $\X$ defined as $\widetilde{X}:= X\times {[B]}$ where $B$ is the unit ball in the dual space $F^*$. 
By Lemma \ref{lem:continuous_map_into_projective_ball}, the set  
$
\TDelta:= \{(x,[v])\in \widetilde{X}: \, x\in \Delta,\,\, T^{*}_x a^{*}(x)v  \neq 0\}
$ is open in $\widetilde{X}$.  In view of Lemmas \ref{lem:continuous_map_into_projective_ball} and \ref{lem:continuous_map_in_norm}  the  maps $\tilde{\varphi}:\TDelta \to \widetilde{X}$  and $\tilde{a}: \TDelta\to [0,+\infty)$
are  continuous.  Hence  $(\X,\tilde{\varphi}, \tilde{a})$ is the continuous linear extension of $(X,\p,aT)$. 
 By Proposition \ref{prop:coc-aT}, $r(a\alpha) 
=  e^{\lambda(aT,\varphi)}$ where $aT(x)=a(x)T_x$, $x\in \Delta$. Thus applying Theorem \ref{thm:varp-principle_for_cocycles} gives $\lambda(aT,\varphi)=\max_{\mu \in {\rm Inv}\left(\TDelta_\infty,\tilde{\varphi}\right)} \, \int_{\TDelta_\infty} 
\ln \tilde{a}(x, [v]) \,\, d\mu = \ \max_{\mu \in {\rm Erg}\left(\TDelta_\infty,\tilde{\varphi}\right)} \,  \int_{\TDelta_\infty} 
\ln \tilde{a}(x, [v]) \,\, d\mu$. In view of Remark \ref{relation-erg-Terg}, essential domains for the systems
$(X\times {[B]},\tilde{\varphi})$ and $(X\times {[S]},\tilde{\varphi})$ are the same. This gives the assertion.
\end{proof}

Now we are ready to describe relationships between  Lyapunov exponents and spectral radius of weighted endomorphisms of $C(X, \B(H))$. The forthcoming result generalizes   variational principle of Latushkin and Stepin 
\cite{LatSt1}, \cite{LatSt}, established in the case where $F=H$ is a Hilbert space, $a$ takes values in compact operators, $\p$ is a homeomorphism and $\alpha_x=id$.

\begin{thm}
 \label{spectral radius main theorem} Let $A = C(X, \B(F))$ where $F$ is a Banach space  and suppose that   $\alpha : A\to A$ is an endomorphism of $A$ such that the generated field of endomorphisms $\{\alpha_x\}_{x\in \Delta}$ consists of isometric inner endomorphism of $\B(F)$.  For every $a\in A$ we have
\begin{equation}
\label{e-Lyap-1}
 \ln r(a\alpha)=\max_{\mu \in {\rm Erg} (\Delta_\infty,\p)}\, \lambda_\mu({aT,\p})
\end{equation}
where $(X,\p)$ is the partial dynamical system dual to $\alpha$, and $aT(x):=a(x)T_x$ for $x\in \Delta$, where $\{T_x\}_{x\in \Delta}\subseteq \B(F)$ a family of isometries such that 
$T_xb=\alpha_x(b)T_x$ for $b\in \B(F)$.  Moreover, for any set $\Omega\subseteq \Delta_\infty$   such that $\mu(\Omega)>0$ for every $\mu \in {\rm Erg} (\Delta_\infty,\p)$, we have
$$
\ln r(a\alpha)= \max_{(x,v) \in \Omega\times S}\lambda_x(C_{T^*a^*,\p}^b,v) 
$$
where $S$ is the unit sphere in $F^*$, and $\Delta \ni x \to T^*a^*(x):=T_x^*a^*(x)\in \B(F^*)$. 
For a fixed $a\in A$, 
the set $\Delta_\infty$ above can be replaced by the essential domain
of  $\varphi:\{x\in \Delta: a(x)\neq 0\}\to X$. 
\end{thm}
\begin{proof} By Lemmas \ref{lem:continuous_map_into_projective_ball} and \ref{lem:continuous_map_in_norm}, 
$(X,\p,aT)$ admits  a continuous linear extension. Hence we get the assertion by   Theorem \ref{thm:variational_principle_Lyapunov} and Corollary~\ref{cor:variational_principle_Lyapunov}. 
\end{proof}
\begin{cor}
\label{c-max-coc}
Retain the notation and assumptions of Theorem \ref{spectral radius main theorem}.  In particular, let $\Omega\subseteq \Delta_\infty$ be  such that $\mu(\Omega)>0$ for every $\mu \in {\rm Erg} (\Delta_\infty,\p)$. Then
\begin{align*}
r(a\alpha)= \max_{x \in \Omega} \lim_{n\to\infty}\Vert C_{aT, \varphi}^f (x,n) \Vert^{\frac{1}{n}}
= \max_{x \in \Omega} \lim_{n\to\infty} \Vert C_{a,\alpha}(x,n) \Vert^{\frac{1}{n}},
\end{align*}
where $C_{a,\alpha}(x,n):=a \cdot \alpha(a) \cdot{\dots}  \cdot\alpha^{n-1} (a) (x)$, cf. Remark \ref{rem:spectral_exponents_and_radii}.
\end{cor}
\begin{proof}
The first equality follows from Theorem \ref{spectral radius main theorem} and the second from \eqref{e-coc-coc1}.
\end{proof}

\begin{rem}
\label{finite-VP} If $F$ is finite dimensional, then every endomorphism $\alpha : C(X, \B(F))\to C(X, \B(F))$ is of the form  \eqref{eq:desired_endomorphism_form}, by  Corollary~\ref{cor-M-n}. 
Thus assumptions of Theorems~\ref{thm:varp-principle-TS}, \ref{spectral radius main theorem} 
are satisfied whenever the induced endomorphisms $\{\alpha_x\}_{x\in \Delta}$ are  isometric (they are necessarily automorphisms).
In particular, if $F=H\cong \C^n$ is a finite dimensional Hilbert space, then \emph{every $*$-endomorphism} $\alpha : C(X, \B(H))\to C(X, \B(H))$ satisfies the assumptions of Theorems~\ref{thm:varp-principle-TS}, \ref{spectral radius main theorem}. 
Moreover, in the Hilbert space  case we have $H^*\cong H$ and thus the above results can be phrased without passing to dual spaces.
\end{rem}

\subsection{Variational principle in the case $A=C(X,D)$}
\label{varp-gen}

In this final subsection we consider the  case when $A=C(X,D)$ and $\{\alpha_x\}_{x\in \Delta}$ are arbitrary endomorphisms of $D$. We show that by an adequate choice of a cocycle with values in $\B(D)$ 
we may reduce the general problem to the situation already treated in  previous sections. 

To this end, let  $\alpha$ be a contractive endomorphism of $C(X,D)$ that generates a partial dynamical system $(X,\p)$. Thus  $\alpha$ is given by formula \eqref{form of endomorphism}
where $\{\alpha_x\}_{x\in \Delta}$ is a continuous field of endomorphisms of $D$.
We also assume that $D$ is a unital  Banach algebra (the unit $1\in D$ has norm one). 
This allows us to treat $D$ as a subalgebra of $\B(D)$. Namely, we have an isometric homomorphism 
$D\ni d \longmapsto \overline{d} \in \B (D)$ where $\overline{d}v := dv$,  for $v\in D$. 
We may extend this embedding to an isometric homomorphism 
$C(X,D)\ni b \longmapsto \overline{b} \in C(X,\B (D))$ where $\overline{b}(x):=\overline{b(x)}$,  for $x\in X$.
Moreover,  each $\alpha_x:D\to D$, $x\in \Delta$,
is contractive and in particular $\alpha_x\in \B (D)$. 
In fact, we may treat  $\{\alpha_x\}_{x\in \Delta}$ as an element $\overline{\alpha}\in C(X,\B(D))$ where
$$
\overline{\alpha}(x):=
\begin{cases}
\alpha_x,& x \in \Delta,\\
0,& x\notin \Delta,
\end{cases}
$$
 We  denote by $T_\p:C(X,\B(D))\to C(X,\B(D))$  an endomorphism associated with  $(X,\p)$:
$$
 T_\varphi(b)(x):=
\begin{cases}
b(\p(x)),& x \in \Delta,\\
0,& x\notin \Delta,
\end{cases}\qquad b\in C(X,\B(D)).
$$ 
Since $\Delta$ is clopen for any $a\in D$  the map $\Delta \ni x \to \alpha_x \overline{a}(\p(x))\in \B(D)$, 
which we denote by  $\overline{\alpha}\cdot \overline{a}\circ \varphi$, can be treated as an element of $C(X,\B(D))$. Formally,  $\overline{\alpha}\cdot \overline{a}\circ \varphi=\overline{\alpha}T_\p(\overline{a})$.
\begin{prop}
\label{prop:spectral_radii_arbitrary_to_inner} 
With the above notation for every $a\in  C(X,D)$ we have 
$$
\ln r(a\alpha)=\ln r(\overline{\alpha}T_\p(\overline{a}) T_\p)=\lambda(\overline{\alpha}\cdot \overline{a}\circ \varphi,\varphi).
$$
That is, the spectral radii of weighted endomorphisms $a\alpha:C(X,D)\to C(X,D)$ and $\overline{\alpha}T_\p(\overline{a}) T_\p:C(X,\B(D))\to C(X,\B(D))$ do coincide, and their logarithms are equal to 
the spectral exponent of the function $\Delta \ni x \to \alpha_x \overline{a}(\p(x))\in \B(D)$ with respect to $(X,\p)$.
\end{prop}
\begin{proof} We have $\ln r(\overline{\alpha}T_\p(\overline{a}) T_\p)= \lambda(\overline{\alpha}\cdot \overline{a}\circ \varphi,\varphi)$ by Remark \ref{rem:spectral_exponents_and_radii}.
By Corollary \ref{cor:spectral_radius_contractive_endomor}, $r(a\alpha)=r(\alpha a)=\lim_{n\to \infty}\| (\alpha a)^n\|^{1/n}$
 where we treat $\alpha a=\alpha(a)\alpha$  as a weighted endomorphism of $C(X,D)$.  By the same corollary we have 
$r(\overline{\alpha}T_\p(\overline{a}) T_\p)
=\lim_{n\to \infty}\|\overline{\alpha}T_\p(\overline{a})\cdot {\dots} \cdot T_\p^{n-1}\Big( \overline{\alpha}T_\p(\overline{a})\Big)\|^{1/n}$. 
Thus it suffices to show that,  for every  $n \in \N$, 
$$
\| (\alpha a)^n\|_{\B (C(X,D))} = \|\overline{\alpha}T_\p(\overline{a})\cdot {\dots} \cdot T_\p^{n-1}\Big( \overline{\alpha}T_\p(\overline{a})\Big)\|_{C(X,\B(D))}.
$$
To this end, we note that for  each $x \in \Delta$ the operator $\alpha_x \overline{a(\varphi(x))}\in \B(D)$ acts according to the formula $
 [\alpha_x \overline{a(\varphi(x))}]d = \alpha_x \big(a(\varphi(x))d\big)$ for $d\in D$. 
Having this in mind  we obtain
\begin{align*}
\| (\alpha a)^n\| &=  \sup_{b\in A, \|b\|=1} \| (\alpha a) \cdots(\alpha a)(\alpha a) b\|_{A}
\\
&= \sup_{b\in A, \|b\|=1} \sup_{x\in \Delta_n} 
\| \alpha_x \Big( a(\varphi (x)) \alpha_{\varphi(x)} \big(a(\varphi^2 (x))
\cdots  \alpha_{\varphi^{n-1}} ( a(\varphi^n (x))
b(\varphi^n (x)))\cdots \big)\Big)  \|_{D}  
\\
&=  \sup_{x\in \Delta_n}  \sup_{d\in D, \|d\|=1}
\| \alpha_x \Big( a(\varphi (x)) \alpha_{\varphi(x)} \big(a(\varphi^2 (x))
 \cdots   \alpha_{\varphi^{n-1}} ( a(\varphi^n (x))
d)\cdots \big)\Big)  \|_{D}  
\\
&=  \sup_{x\in \Delta_n}  \sup_{d\in D, \|d\|=1}
\|  \Big(\alpha_x \overline{a(\varphi(x))}\cdots   \alpha_{\varphi(x)} \overline{a(\varphi^2(x))}  \cdots\alpha_{\varphi^{n-1}(x)} \overline{a(\varphi^{n} (x))}\Big) d  \|_{D}  
\\
&=  \sup_{x\in \Delta_n}  
\|  \alpha_x \overline{a(\varphi(x))}   \alpha_{\varphi(x)} \overline{a(\varphi^2(x))}  \cdots \alpha_{\varphi^{n-1}(x)} \overline{a(\varphi^{n} (x))} \|_{\B(D)}  
\\
& =  \|\overline{\alpha}T_\p(\overline{a})\cdot {\dots} \cdot T_\p^{n-1}\Big( \overline{\alpha}T_\p(\overline{a})\Big)\|_{C(X,\B(D))}.
\end{align*}
\end{proof}
The theory developed in previous sections give formulas for $r(\overline{\alpha}T_\p(\overline{a}) T_\p)=e^ {\lambda(\overline{\alpha}\cdot \overline{a}\circ \varphi,\varphi)}$,
which in  view of Proposition \ref{prop:spectral_radii_arbitrary_to_inner}  are also formulas for  $r(a\alpha)$. For the sake of completeness we state them explicitly:

\begin{thm}
\label{thm:varp-principle} 
 Let   $\alpha : A\to A$ be a contractive  endomorphism of $A=C(X,D)$ generating a partial dynamical system $(X,\p)$ and let $a\in A$. 
Let $\Omega\subseteq \Delta_\infty$ be any set  such that $\mu(\Omega)>0$ for every $\mu \in {\rm Erg} (\Delta_\infty,\p)$. Then
$$
\ln r(a\alpha)=\max_{\mu \in {\rm Erg} (\Delta_\infty,\p)}\, \lambda_\mu(\overline{\alpha}\cdot \overline{a}\circ \varphi,\varphi)=\max_{(x,v) \in \Omega\times S}\lambda_x(C_{(\overline{a}\circ \varphi)^*\overline{\alpha}^*,\p}^b,v) 
$$
 where $\overline{\alpha}\cdot \overline{a}\circ \varphi$ denotes the function $\Delta \ni x \to \alpha_x \overline{a}(\p(x))\in \B(D)$, $(\overline{a}\circ \varphi)^*\overline{\alpha}^*$ denotes the function
 $\Delta \ni x \to \overline{a(\varphi (x))}^* \alpha_x^*\in \B(D^*)$ and $S$ is the unit sphere in $D^*$.
 
\end{thm}
\begin{proof} In  view of   Proposition \ref{prop:spectral_radii_arbitrary_to_inner}, the 
assertion follows from Theorem \ref{spectral radius main theorem} applied to 
a weighted endomorphism with a weight equal to 
$\overline{\alpha}\cdot \overline{a}\circ \varphi$ and trivial field of endomorphisms of $\B(D)$. 
Alternatively, one can apply Corollary~\ref{cor:variational_principle_Lyapunov},  
Theorem~\ref{thm:variational_principle_Lyapunov} and Remark \ref{relation-erg-Terg}.
\end{proof}

\begin{rem}
Variational principles obtained in  Theorem~\ref{thm:varp-principle} exploit the  operator algebra $\B(D)$ and the dual space $D^{*}$ of a Banach algebra $D$. As a rule these spaces are rather bulky. For example, if $D = \B (F)$ for some Banach space $F$ then $\B(D)=\B(\B(F))$ and $D^*=\B(F)^*$ which are huge in comparison to $\B(F)$ and $F^*$,  respectively. 
Thus Theorems~\ref{thm:varp-principle-TS}, \ref{spectral radius main theorem} are much more efficient, but we can apply them only under the assumption that the associated field 
  $\{\alpha_x\}_{x\in \Delta}$ consists of  isometric inner endomorphisms of $\B(F)$.
\end{rem}

\section{Concluding remarks and potential applications}\label{sec:final}
The variational principles established in this paper have
a flavor of  Dinaburg-Goodman variational principle 
linking the topological entropy with measure theoretical entropies,
and more generally Ruelle-Walters variational principle \eqref{varp-topolo-pressure}
expressing the topological pressure is terms of free energy. 
These celebrated results have a deep theoretical meaning. 
We have expressed spectral radii of a wide class of operators 
in terms of measure Lyapunov exponents for ergodic measures for a given topological dynamical system.
Thus our formulas establish
strong explicit link between theory of operators and dynamical systems.
They are far reaching generalizations,  and in the noncommutative case improvements, of a number of previous results of this type, see
\cite{Kitover}, \cite{Lebedev79}, \cite{LatSt1}, \cite{LatSt}, \cite{Anton_Lebed}, \cite{ChLat}. 
 One of the immediate interesting consequences is
the improved Gelfand's formula which states that the spectral exponent of every operator is 
a Lyapunov exponent in some direction (Corollary \ref{cor-Gelf}).
We believe that the results of the present paper can be
used to prove a very general continuous Multiplicative Ergodic Theorem for
arbitrary (not necessarily quasi-compact) cocycles taking values in   operators on a reflexive Banach space, cf. Remark \ref{rem:MET}.

Calculation of ergodic measures  is one of the most important
problems in ergodic theory. As a rule, for a non-trivial irreversible dynamical system $(X,\varphi)$, 
invariant measures 
${\rm Inv}(X,\varphi)$   form a Poulsen symplex,
which means that the set of ergodic measures ${\rm Erg}(X,\varphi)$ is large.
Therefore in general one can not expect our formulas to give immediate
and explicit answers. Nevertheless, the presented 
theory gives a clear procedure how to approach the problem  
and significantly simplifies general calculations 
(one of the main improvements is the change from $\lim_{n\to \infty}\sup_{x\in X}$ to
$\max_{x\in X} \lim_{n\to \infty}$, cf. Corollary \ref{cor:ergodic_sums})
We illustrate this on some examples. We start with  operators associated to weighted endomorphisms which are not 
far from being compact. Symptomatically, it seems that the only `explicit' results about spectra of weighted composition operators on 
 $C(X, \B(F))$ concern the compact case, see \cite{Jamison}, \cite{Kamowitz}. We generalize formulas for 
spectral radii appearing there:

\begin{prop}\label{prop:compact_spectral_radius}
Suppose that $aT\in \B(E)$, $a\in A:=C(X, \B(F))$, are   abstract weighted shift operators such that the  associated endomorphism $\alpha:A\to A$  consists of isometric inner endomorphisms. 
So that $\alpha$ is of the form 
\eqref{eq:desired_endomorphism_form} for a partial map $\varphi:\Delta \to X$ and a field of isometries   $\{T_x\}_{x\in \Delta}\subseteq \B(F)$. Suppose also that $a\in A$ is such that
\begin{enumerate}
\item\label{enu:compact1}  $
F\subseteq \Delta^a:=\{x\in X: a(x)\neq 0\} \text{ is compact }\Longrightarrow |\varphi(F)| <\infty;
$
\item\label{enu:compact2} $a(x)\in \K(F)$ is compact for every $x$ in the essential domain $\Delta^a_{\infty}$  for $\varphi:\Delta^a\to X$
\end{enumerate}
 (these are necessary, but not sufficient, conditions for $a\alpha:A\to A$ to be compact, cf.  \cite{Jamison}).
Then putting $W(x):=a(x)T_x$, $x\in \Delta$, we have
$$
r(aT)=r(a\alpha)=\max_{x\in \Delta^a_{\infty}, \atop
\varphi^n(x)=x, n \in \N} \sqrt[n]{\|W(x)W(\varphi(x))...W(\varphi^{n-1}(x))\|}.
$$
The norm of each of the operators $W(x)W(\varphi(x))...W(\varphi^{n-1}(x))$, in the above formula,
realizes on an eigenvector.
\end{prop}
\begin{proof}
By \cite[Lemma]{Jamison} condition \eqref{enu:compact1} holds if and only if every connected component  of $\Delta^a$
is contained in an open set $V\subseteq \Delta$ such that $\varphi|_V$ is constant. 
This readily implies that 
the set of non-wandering points for $\varphi:\Delta^a\to X$ consists only of periodic points.
Thus every ergodic measure  is supported
on a periodic orbit. That is, 
${\rm Erg} (\Delta_\infty^a,\p)=\{\frac{1}{n}\sum_{k=0}^{n-1}\delta_{\varphi^k(x)}: \varphi^n(x)=x\in \Delta_\infty^a, n \in \N\}$. 
Hence by Theorem \ref{spectral radius main theorem} we get
$$
r(aT)=r(a\alpha)=\max_{x\in \Delta^a_{\infty}
, \atop
\varphi^n(x)=x, n \in \N}\lim_{k\to \infty} \sqrt[k]{\|W(x)W(\varphi(x))...W(\varphi^{k-1}(x))\|}.
$$
Using the submultiplicativity of operator norm, for every  $k\in \N$ and $x\in \Delta^a_{\infty}$ with $\varphi^n(x)=x$  we get $\|W(x)W(\varphi(x))...W(\varphi^{nk-1}(x))\|^{\frac{1}{kn}}\leq\|W(x)W(\varphi(x))...W(\varphi^{n-1}(x))\|^{\frac{1}{n}}$.
By  \eqref{enu:compact2} this inequality is in fact 
an equality. Indeed,  since $W(x)...W(\varphi^{n-1}(x))$
is compact,  there is a norm one eigenvector $v\in F$ 
for $W(x)...W(\varphi^{n-1}(x))$ such that
\begin{align*}
\|W(x)W(\varphi(x))...W(\varphi^{n-1}(x))\|^{\frac{1}{n}}
&=\|W(x)W(\varphi(x))...W(\varphi^{n-1}(x)) v\|^{\frac{1}{n}}
\\
&=\|W(x)W(\varphi(x))...W(\varphi^{nk-1}(x))v\|^{\frac{1}{kn}}.
\end{align*}
This reduces the above formula for $
r(aT)=r(a\alpha)$ to the one stated in the assertion.
\end{proof}
The above proposition and the following example indicate that 
when the dynamics of $(\Delta_\infty^a,\p)$ is not complicated or not
far from being reversible, then there is a chance for more  explicit formulas
for the spectral radius.
\begin{ex} Let us consider generalisations of compression of unitaries studied in \cite{kwa-logist}.
Namely, we put $E=L^p(\R_+,F)$, where $p\in [1,\infty]$, $F$ is a Banach space and $\R_+=[0,\infty)$ is equipped with the Lebesgue measure.
We fix a continuous monotonically increasing function $\gamma:\R_+\to \R_+$  satisfying the condition 
$
\gamma(t+1)=\gamma(t)+1$, $t \in \R_+$, and use it to define an operator
$T_p\in \B(E)$ by the formula  
$$
(T_pf)(t):=|\gamma'(t)|^{\frac{1}{p}}f(\gamma(t)), \qquad f\in E=L^p(\R_+,F),\,\, t\in \R^+;
$$
 $\gamma'(t)$ exists almost everywhere because  $\gamma$ is monotone and we put $\frac{1}{\infty} =0$
when $p=\infty$.
Note that $T_p$ is not invertible, unless $g(0)=0$. It is adjoint, in the sense of Definition \ref{definicja czesciowej izometrii},
to an isometry given by $(S_pf)(t):= |(\gamma^{-1})'(t)|^{\frac{1}{p}}f(\gamma^{-1}(t))$,  for $t \in [\gamma(0),\infty)$
and $(S_pf)(t):=0$ for $t\in [0,\gamma(0))$. 
We let $A$ to be the algebra of operators of multiplication by periodic functions in $C(\R_+, \B(F))$ with period $1$. 
Then $A\cong C(S^1,\B(F))$ and one readily sees that $aT_p$, $a\in A$, are abstract weighted shifts associated 
with the authomorphism $\alpha:A\to A$ given by composition with the homeomorphism
$\varphi:S^1\to S^1$ where $\varphi(e^{2\pi i t})=e^{2\pi i \gamma(t)}$,  $t\in [0,1)$.
We recall that the \emph{rotation number} for $\varphi$ is defined as the fractional part
$$
\tau(\varphi):= 
\text{frac} \left(
\, \lim_{n\to \infty}\frac{\gamma(t)}{n} \right) 
\in [0,1)
$$
where $\lim_{n\to \infty}\frac{\gamma(t)}{n}$  does not depend on  $t\in \R$, and the number 
$\tau(\varphi)$ depends only on $\varphi$. 
We have the following cases, see for instance \cite[Section 7]{Brin}:
\begin{enumerate}
\item If
$\tau(\varphi)=\frac{m}{n}$, where  $\frac{m}{n}$ is an irreducible fraction, then  $\Omega(\varphi)=\{x\in S^1: \varphi^n(x)=x, \, n\in \N\}$
and therefore ${\rm Erg}(S^1,\varphi)=\{\frac{1}{n}\sum_{k=0}^{n-1}\delta_{\varphi^k(x)}: \varphi^n(x)=x\in S^1,  n \in \N\}$.
In particular, if we form a set $\Omega$ by choosing a single  point from every periodic orbit then
$
r(aT_p)=r(a\alpha)=\max_{x\in \Omega} \lim_{k\to \infty} \sqrt[kn]{\|\left(a(x)a(\varphi(x))...a(\varphi^{n-1}(x))\right)^k\|}.
$
And if $a$ attains values in compact operators, then similarly as in the proof of Proposition \ref{prop:compact_spectral_radius}
we get
$$
r(aT_p)=r(a\alpha)=\max_{x\in \Omega}  \sqrt[n]{\|a(x)a(\varphi(x))...a(\varphi^{n-1}(x))\|}.
$$
\item If  
$\tau(\varphi)\not \in \mathbb{Q}$ is irrational, then  $\varphi$ is uniquely ergodic, that is  ${\rm Erg}(S^1,\varphi)=\{\mu\}$, and either 
$\Omega(\varphi)=S^1$ in which case $\mu$ is equivalent to the Lebesgue measure on $S^1$, or $\Omega(\varphi)$ is homeomorphic to a Cantor 
set and then $\mu$ is equivalent to the singular measure with Cantor distribution. 
In the first case with  we have $\lambda_\mu({a,\p})=\lambda_x(C_{a,\p}^f)$ for almost  any $x\in S^1$. Thus we may choose a point $x\in S^1$ randomly and with probability one we get
$$
r(aT_p)=r(a\alpha)=\lim_{n\to \infty} \sqrt[n]{\|a(x)a(\varphi(x))...a(\varphi^{n-1}(x))\|}.
$$
In the second case, one needs to first detect the Cantor set $\Omega(\varphi)$ and then check
the above limits in points of  $\Omega(\varphi)$. Fortunately, by Denjoy theorem \cite[Theorem 7.2.1]{Brin} this singular second case can not occur if 
$\gamma$ (or equivalently $\varphi$) is of $C^2$-class  (in fact it suffices that $\gamma'$ is continuous and  of bounded variation).
\end{enumerate}
\end{ex}

Non-trivial complicated irreversible dynamics come from expanding maps. 
If for a given $a$ the restricted  map
 $\varphi:\Delta^a_{\infty}\to \Delta^a_{\infty}$ is expanding,  then so is the map 
extended $\tilde{\varphi}:\TDelta_{\infty}\to \TDelta_{\infty}$ in the linear extension. Then the dynamics and ergodic measures
for both of the systems $(\Delta^a_{\infty}, \varphi)$ and $(\tilde{\varphi},\TDelta_{\infty})$ 
can be analyzed  using symbolic dynamics -  fixing Markov partitions we may reduce
the analysis to topological Markov chains, cf. \cite{Bowen}, \cite{Ruelle0}, \cite{Ruelle89}. 
The \emph{topological Markov chain} 
with the set of states $\{1,..,n\}$ and transition matrix $\mathbb{A}=[t_{i,j}]_{i,j=1}^n$, $t_{i,j}\in \{0,1\}$, 
 is the map 
$\sigma_{\mathbb{A}}:\Sigma_{\mathbb{A}}\to \Sigma_{\mathbb{A}}$ where
$
\Sigma_{\mathbb{A}}:=\{(\xi_1,\xi_2,...)\in \{1,...,n\}^\N: t_{\xi_i,\xi_{i+1}}=1\}$ 
and
$\sigma_{\mathbb{A}} (\xi_1,\xi_2,\xi_3,...):=(\xi_2,\xi_3,...).
$
The elements  of ${\rm Inv}(\Sigma_{\mathbb{A}},\sigma_{\mathbb{A}})$
are in one-to-one correspondence with  collections of nonnegative numbers $\{p_{\xi_1,\xi_2,...\xi_k}: (\xi_1,\xi_2,...\xi_k)\in\{1,...,n\}^{k}, k\in\Z\}$ such that $\sum_{i=1}^n p_i=1$,  
 $$
p_{\xi_1,\xi_2,...\xi_{k+1}}=\sum_{i=1}^n p_{\xi_1,\xi_2,...\xi_k,i}=\sum_{i=1}^n p_{i,\xi_2,\xi_3,...\xi_{k+1}}
$$
 and $p_{\xi_1,\xi_2,...\xi_{k}}>0$  only if $(\xi_1,\xi_2,...\xi_{k})$ is a path in the directed graph given by $\mathbb{A}$,
cf., for instance, \cite{Walters}.
Using this, and perhaps employing  computer calculations, one can obtain good lower bounds for the corresponding spectral radius,
but the thorough discussion of these issues is beyond the scope of the present paper. 
In the simplest situation, when  $A=C(\Sigma_{\mathbb{A}})$,  $\alpha(a):=a\circ \sigma_{\mathbb{A}}$, and $a\in A$ depends only 
on $N$ first coordinates: $a(\xi_1,\xi_2,...)= a(\xi_1,\xi_2,...,\xi_N)$ for all $(\xi_1,\xi_2,...)\in \Sigma_{\mathbb{A}}$, then
$$
r(a\alpha)=\max_{\{p_{\xi_1,\xi_2,...\xi_{N}}\}} 
\prod_{(\xi_1,\xi_2,...,\xi_{N})\in \Sigma_{\mathbb{A}}^N} |a( \xi_1,\xi_2,...\xi_{N})|^{p_{\xi_1,\xi_2,...,\xi_{N}}}
$$
where $\Sigma_{\mathbb{A}}^N:=\{(\xi_1,\xi_2,...,\xi_{N}): (\xi_1,\xi_2,...)\in \Sigma_{\mathbb{A}} \}$ is 
a finite set and the maximum is taken over all positive numbers $\{p_\xi\}_{\xi \in \Sigma_{\mathbb{A}}^N}$ 
such that 
$\sum_{\xi \in \Sigma_{\mathbb{A}}^N} p_{\xi}=1$.

\end{document}